\documentclass[letterpaper, 10 pt, conference]{ieeeconf}  

\IEEEoverridecommandlockouts                             

\usepackage{graphicx}
\usepackage{booktabs} 
\usepackage[utf8]{inputenc}

\usepackage{amsmath, amssymb, amsfonts, amsthm}

\usepackage{color}
\usepackage{bm}
\usepackage{hhline}
\usepackage{mathabx}
\usepackage{multirow}
\usepackage{setspace}
\usepackage{algorithm}
\usepackage{setspace}
\usepackage[noend]{algpseudocode}
\usepackage{algpseudocode}
\usepackage{mathtools}

\newtheorem{theorem}{Theorem}

\newtheorem{remark}[theorem]{Remark}
\newtheorem{definition}{Definition}

\newtheorem{assumption}{Assumption}

\newtheorem{lemma}{Lemma}
\newtheorem{proposition}{Proposition}
\newtheorem{corollary}{Corollary}

\usepackage[noadjust]{cite}

\usepackage{todonotes}

\title{\LARGE \bf
From Spectral Theorem to Statistical Independence with Application to System Identification }

\author{Muhammad Abdullah Naeem, Amir Khazraei and Miroslav Pajic 
\thanks{Muhammad Abdullah Naeem, Amir Khazraei and Miroslav Pajic are with the Department of Electrical and Computer Engineering, 
        Duke University, Durham, NC 27708, USA, Email: 
        {\tt\small muhammad.abdullah.naeem@duke.edu}}%
}

\begin{document}
\maketitle

\begin{abstract}
High dimensional random dynamical systems are ubiquitous, including-but not limited to- cyber-physical systems, daily return on different stocks of S\&P 1500 and velocity profile of  interacting particle systems around McKeanVlasov limit. Mathematically, underlying phenomenon can be captured via a stable $n-$ dimensional linear transformation `$A$' and additive randomness. System identification aims at extracting useful information about underlying dynamical system, given a length $N$ trajectory from it (corresponds to an $n \times N$ dimensional data matrix)  We use spectral theorem for non-Hermitian operators to show that spatio-temperal correlations are dictated by the \emph{discrepancy between algebraic and
geometric multiplicity of distinct eigenvalues}  corresponding to state transition matrix. Small discrepancies imply that original trajectory essentially comprises of multiple \emph{lower dimensional random dynamical systems living on $A$ invariant subspaces and are statistically independent of each other}. In the process, we provide first quantitative handle on decay rate of finite powers of state transition matrix $\|A^{k}\|$ . It is shown that when a stable dynamical system has only one distinct eigenvalue and discrepancy of $n-1$: $\|A\|$ has a dependence on $n$, resulting dynamics are \emph{spatially inseparable} and consequently there exist at least one row with covariates of typical size $\Theta\big(\sqrt{N-n+1}$ $e^{n}\big)$ i.e., even under stability assumption, covariates can \emph{suffer from curse of dimensionality }. In the light of these findings we set the stage for non-asymptotic error analysis in estimation of state transition matrix $A$ via least squares regression on observed trajectory by showing that element-wise error is essentially a variant of well-know Littlewood-Offord problem.
\end{abstract}


\section{Introduction} 
\label{sec:intro}
Successful completion of any control task, relies heavily on availability of accurate models for underlying dynamical system. Owing to advances in machine learning and high dimensional statistics,  when physics based models are too complicated or do not really exist to begin with as in the case of S \& P 1500, one attempts to use some sort of regression models on an observed trajectory of underlying dynamical system of length $N$ and provide a non-asymptotic high probability guarantees on learning its' correct models. 
Said in another way, 
given an $n \times N$ dimensional data matrix which captures evolution over time of $n-$ dimensional random dynamical system, goal of system identification is to extract `useful' information about underlying system. 

To name a few recent results:  \cite{oymak2021revisiting} provides non-asymptotic analysis on learning system parameters (state transition matrix, input-to-state transformation, etc) of a linear time invariant (LTI) system from a single observed trajectory of input-output data. \cite{simchowitz2018learning},\cite{sarkar2019near} and \cite{faradonbeh2018finite} provides statistical analysis on learning the state transition matrix via Ordinary Least Squares~(OLS) estimator on an observed trajectory $(x_{t})_{t=1}^{N}$, where dynamics follow $x_{k+1}=Ax_{k}+w_{k}$, dimension of the underlying state space is $n$ and $w_{k}$ is isotropic Gaussian, 
\cite{tu2018least} provides non-asymptotic analysis on learning value function corresponding to closed loop stable policy for a Linear System and \cite{tsiamis2022online} provides non-asymptotic analysis for online Kalman filtering.

Surprisingly, their has been no reference to how dimensionality of underlying space plays out in learning tasks. As additive randomness is usually assumed to be isotropic, how does spectrum of underlying linear transform translates into spatio-temeperal correlations of the covariates represented in the data matrix. The gap might have stemmed from the fact that these results classify dynamical systems only based on the magnitude of their eigenvalues, which suffices if the underlying state-transition matrix is Hermitian. To fully characterize any linear transformations whether or not it is stable in control theoretic sense, one needs to take into account geometric part of the spectrum as well i.e., span of eigenvectors. Furthermore,
as high probability guarantees in majority of recent work requires sufficient length of simulated trajectory to be lower bounded by various Grammians(essentially higher degree polynomial of $A$) and prior to this work uncertaintly loomed over \emph{sensitivity of operator norm to similarity transforms}-so it does not come as much of a surprise why these dependencies on dimensionality had been missing and spatial interactions were never before even a part of discussion.

On the quantiative side of things, regardless of some minor technicalities, starting point in all these problems assume that via some oracle we are given dimension of the underlying state space $n$ and spectral radius $\rho(A)<1$ (i.e., underlying dynamical system is stable) then a regression  algorithm is proposed and if the length of simulated trajectory is of order $\log\bigg(\frac{1}{\rho(A)}\bigg)$ then the desired accuracy is achieved.
A common theme in all these works is controlling decay of power of $\|A^{k}\|$ via Gelfands' formula. \cite{oymak2021revisiting} learns systems Markov Parameters up till time $T$ such that $\sum_{\tau=T}^{\infty}\|A^{\tau}\|= O(1)$ and then estimate system parameters using Ho-Kalman algorithm \cite{ho1966effective}. Pretty much exclusively, all these result argue along the same lines \cite{simchowitz2018learning,tu2018least,oymak2021revisiting}:  if $\rho(A)<1$ then $\|A^{k}\| \leq$ poly$(k) \rho(A)^{k}$ (poly$(k)$ is used to denote polynomial function of $k$),   
and using tools from high dimensional statistic derive bounds on the length of the trajectory in terms of spectral radius, various Gramians related to system dynamics for the desired accuracy as discussed in preceding~paragraph.
 However, Gelfands' formula is in fact true for infinite dimensional Hilbert spaces as well, and we are dealing with finite dimensional linear transformations which leaves a lot of 
 room for imporvement in decay rate of $A^{k}$. 
 


These limitations suggest we break down finite sample analysis problem into two components, dynamics part and probabilistic part. So, we take an initial step in developing a non-asymptotic, geometric and interpret-able version of systems theory that can explain dynamical evolution of high dimensional systems by studying $(i)$ ~quantitative decay of $A^{k}$, with explicit dependence in dimensionality of the state space and the number of iterations, $k$. In order to get this handle we first had to qualitatively differentiate linear transformations i.e., ones' with strong or weak spatial couplings. We then combine our results with findings of high dimensional geometry/ statistics to collect $(ii)$~ various estimates(explicit in $n$ and $N$) that will be useful for complete understanding  of least squares regression when trying to estimate high dimensional LTI systems from a single observed trajectory and in fact show qualitatively different behavior of LTI systems(with same eigenvalues) lead to qualitative different behavior of regression.


\subsection{Related Work}
To the best of authors' knowledge, only very recently, similar questions appeared in \cite{tsiamis2021linear} and \cite{tsiamis2022learning}, where they posed the question of what kind of linear systems are hard to learn and control and also provided with real world examples. However they could not mathematically classify such systems and  complexity of learning was given itself as a function of operator norm of the the linear transform. Unfortunately, their analysis only works under the assumption that operator norm is independent of dimensionality of the state space, which we will show is not true for spatially inseparable stable systems. 
On the quantitative side of things, in finite sample analysis of temporal difference learning for value function corresponding to a closed loop stable policy given in~\cite{tu2018least} quantifies the decay of the operator norm via $\mathcal{H}_{\infty}$ norm of the resolvent of state transition matrix, without any interpretability:  define $\Phi_{A}(z):= (zI-A)^{-1}$ for  complex numbers $z$ then their result claims for any $\rho \in (\rho(A),1)$ and for all $k  \in \mathbb{N}$, $\|A^{k}\| \leq \sup_{z \in \mathbb{C}: |z|=1 } \| \Phi_{\rho^{-1} A}(z) \| \rho^{k}$, but using Neumann series expansion of resolvent $\|\Phi_{\rho^{-1}A}(z)\| = \| \frac{1}{z} \sum_{l=0}^{\infty} ( \frac{A}{z \rho})^{l} \|$ so their result essentially translates to $\|A^{k}\| \leq  \sup_{z \in \mathbb{C}: |z|=1 }  \| \frac{1}{z} \sum_{l=0}^{\infty} \bigg( \frac{A}{z \rho}\bigg)^{l} \| \rho^{k}$ i.e., their quantitative handle controls norm of $\|A^{k}\|$ by a polynomial of all the finite powers of $A$ weighted by some constant factor.
\cite{simchowitz2018learning} provides control on norm of $\|A^{k}\|$ by computing norm of the associated Jordan block but these estimates require norm bounds on associated similarity transformation and its' inverse(we will show is not necessary) and do not offer any interpretation into how size of the Jordan blocks translate into operator norm of powers~of~$A$. 

For various variants of regression problem related to dynamical systems, existing literature focuses exclusively on upper bounding operator norm by some martingale term and showing least singular value of the data matrix is lower bounded with high probability. Concentration behavior of martingale terms and various quadratic forms in literature are studied using $\emph{Hanson Wright inequality}$, which shows deviation of quadratic form based on Frobenius and operator norm of the weights defining it and high probability estimates are given as a function of these norms. As we show in Section \ref{sec:concmdistsigma1}, while studying the concentration behavior of distance between a fixed $n-1$ dimensional subspace of $\mathbb{R}^{N}$ and a trajectory of length $N$ from one dimensional ARMA model, that distance function is essentially a quadratic form with Frobenius and operator norm of its weights having potential dependence on $N$ (number of iterations) and $n$(dimension of underlying state space). Hidden dependence on dimensionality of state space and number of iterations may suggest existing upper bounds are vacuous. Furthermore, lower bounding the error would require quantifying largest singular value of the data matrix, which had not been explored before this work. 
\subsection{Contributions and Results:}
The contribution of this work is twofold. First, is bounding the norm of  $k-$th power of a stable state transition matrix. Most of the existing work assume that stable matrices converges exponentially fast to $0$ and the rate is dependent on spectral radius. However, such arguments can hide dependency on the size of underlying state space. Limitation of existing bounds arise from ignoring the fact:  given oracle bound on spectral radius $\rho<1$ and underlying dimension of the state space $n$, there exists a class of linear transformations that satisfy spectral radius restriction but elements of this class can have their $k-$ th powers' norm behave very differently from each other; an $n \times n$ matrix with diagonal element of $\rho$ is similar to $n$ one dimensional systems and $\|A^{k}\|=\rho^{k}$, on the other hand we will show that there exists state transition matrix with same spectral radius, but its' resulting dynamics have strong spatial correlations and $\|A^{k}\| \geq 1$ for $k=n-1$. Recognizing these issues, we provide answer to the following question:  

\emph{Question: Given dimensions of the underlying state space $n$ and a constant $\rho \in (0,1)$, provide uniform bounds w.r.t dimension of underlying state space for the smallest $\Gamma(n,\rho) \in \mathbb{N}$ such that any linear transformation $A \in \mathbb{C}^{n \times n}$ with spectral radius $\rho(A)$ being equal to $\rho$ satisfy $\|A^{k}\|< 1$  for all $k > \Gamma(n,\rho)$}?

Vaguely speaking, for a fixed $n$ and $\rho$ an upper bound with a tight lower bound adds a notion of entropy/ size to classes of linear transformation with given spectral radius, as quantiatively different behavior should stem from qualitatively different nature of linear transformation. We show that $\Gamma(n,\rho) = O\bigg(\frac{(n-1) \ln{n}}{\ln{\frac{1}{\rho}}}\bigg)$, and in fact our bounds are tight up to $\ln{n}$ factor.
Consequently, revealing potential hidden dependencies on dimensionality of underlying state space in sample complexity for various learning algorithms in literature. These novel bounds are derived by using using spectral theorem for non-Hermitian finite dimensional linear transformations and equipping underlying space with inner product structure. Position of the eigenvalues  and characteristic polynomial only correspond to algebraic structure of Linear operator; Geometric structure follows from the span of associated eigenvectors. Given each distinct eigenvalue of $A$, discrepancy between its' algebraic and geometric multiplicity leads to an $A-$ invariant subspace with dimension equal to size of the discrepancy.  By restricting $A$ onto each invariant subspace via projection operator and computing the norm of the powers of $A$ on the underlying subspace, allows us to control the overall norm of the powers of $A$. We would like to highlight that our result negates the common misconception of operator norm being sensitive to choice of basis. 

Second contribution is on the error analysis for estimating linear time-invariant(LTI) systems from a single observed trajectory via methods of ordinary least squares(OLS). 
Although, system identification via ordinary least squares regression had been a hot topic of research for last few years see e.g., \cite{sarkar2019near}, \cite{simchowitz2018learning}, and \cite{faradonbeh2018finite}. However, recently it was noted in \cite{naeem2023high} and \cite{tsiamis2021linear} that there exists example of stable dynamical systems in dimension $\geq 10$ where OLS contains non-vanishing error, pointing out gaps in existing analysis. Something that has been left  un-noticed: least squares/ regression related problem are of geometric nature for example with minimal effort we manage to point out its' deterioration with possible dependence between rows of the data matrix, which happens in case of large discrepancy between algebraic and geometric multiplicity of eigenvalues associated with state-transition matrix. Using inner product structure of the state space and sample space we manage to get an almost `closed' form expression for element-wise least squares error and turns out to be a scalar weigthed random walk of standard Gaussians with weights defined by the columns of pseudo-inverse of the data matrix, where columns of pseudo-inverse are constrained to be orthonormal to the rows of the data-matrix and hence capture structural properties of the entire trajectory show qualitatively different behavior of pseudo-inverse in the presence of strong spatial correlations versus only temporal correlations. 
Error is  at most a polynomial function of all the Gaussian excitations of the dynamical system weighted by powers of $A$, which is a higher degree variant of \emph{Littlewood-Offord problem}, a work in progress and requires an entire paper of its' own. Furthermore, $i-$th diagonal entry of the inverse sample covariance matrix (which dictates the $\ell_{2}$ norm of the $i-$ th column of pseudo-inverse) is inverse distance squared between the $i-$th row of the data matrix and hyperplane spanned by all the other rows of the data matrix, which brings us to \emph{concentration of measure phenomenon}. Quantifying distance between a fixed subspace and random point in $N$ dimensions, where every entry is independent of each other has been extensively studied before and lead to remarkable progress in research on Random Matrix theory, see (e.g., \cite{tao2005random,tao2010random,rudelson2008littlewood,rudelson2009smallest} ). We extend it to the case when random vector is trajectory is a  from a Markovian data, by computing trace of the higher powers of covariance matrix of length $N$ Markovian trajectory in order to get tighter control of operator norm of the covariance matrix that we can combine with \emph{Talagrands' Inequality} to get concentration results for the distance between length $N$ random vector and an $n-1$ dimensional subspace of $\mathbb{R}^{N}$. 
We also show how to provide high probability upper bound on the largest singular value of data matrix when provided with control on the decay of finite powers of $A$. 

\subsection{Paper structure }
The paper is organized as follows. In Section \ref{sec:Not and Prelim}, we introduce notation and preliminaries. In Section \ref{sec:FDL-Base-free}, we provide a brief introduction to basis-free approach to Linear Transformation. After introducing an inner product structure on domain and image space of Linear transformation, adjoint operator is also introduced.  Section \ref{sec:invsub} aims at familiarizing the reader with spectral decomposition of finite dimensional linear operators; how \emph{discrepancy} between algebraic and geometric multiplicity of eigenvalues corresponding to a given linear transform leads to   generalized eigenvectors and invariant subspaces. Consequently, using  projection operators we can decompose original linear transform into multiple lower dimensional linear operators which can be thought as independent of each other so controlling $k-$ th power norm of each lower dimensional linear operator provides a sharp control on the $k-$ th power norm of original operator and depends non-asymptotically on discrepancy of each distinct eigenvalue along with its magnitude, as shown in Section \ref{sec:kpow}. Section \ref{sec:tensorization}, marks the begining of second half of the paper where probabilistic side of things enter the picture. Heavy emphasis is put towards familiarizing the reader to concentration of measure of phenomenon; particularly important is \emph{dimension independent tensorization of Talagarands' inequality} for norm stable diagonalizable dynamical systems with Gaussian excitations. We will eventually show that exact, elementwise error in OLS is a weighted sum of standard normals and its' concentration behavior is studied in literature under the well-known Littlewood-Offord problem. Therefore, we briefly highlight their findings in particular how concentration behvaior can depend on structure of weigths  .We begin Section \ref{sec:sys-id-diag}, with general least squares estimation problem and link the estimation error with linear dependence of the given basis. Subsequently, OLS set up for Linear Dynamical systems with Gaussian noise is introduced. Two different error bounds in estimation are introduced, a geometric one based on controlling distance between a given row and conjugate hyperplane and another one based on extreme singular values of the data matrix. Concluding the section with highlight result of how spectral theorem implies statistical independence for row blocks of the data matrix. An `almost' closed form solution to element-wise least-squares error is provided in Section \ref{sec:exactOLS} in terms of inner product between rows of Gaussian ensemble and  column of pseudo inverse. Dependence of error on the structure of co-efficients of pseudo-inverse is highlighted by linking it to the Littlewood-Offord problem. Qualitatively different structure of constraints are shown for data generated from state transition matrix with same eigenvalue, one with no spatial correlations and other being \emph{strongly spatially correlated}. In Section \ref{sec:concmdistsigma1}, we quantify the concentration behavior of distance between various length $N$, one dimensional stable ARMA trajectories and $n-1$ dimensional subspace of $\mathbb{R}^{N}$; with our novel control techniques on spectrum of the covariance matrix associated with length $N$ stable one dimensional trajectory. To the best of authors' knowledge this is the first time concentration behavior is explicitly provided with respect to length of simulated trajectory and dimensions of the underlying state space. Furthermore, we also provide first quantitative handle on largest singular value of the data matrix in terms of powers of $A$ and spectrum of rectangular Gaussian ensemble, typical behavior of largest singular value follows from Talagrands' inequality. Conclusive remarks and future work is discussed in Section \ref{sec:conclusion}.  

\section{Notation and Preliminaries}
\label{sec:Not and Prelim}
\vspace{-2pt}
\subsubsection{Notation}  We use ${I}_{n}\in\mathbb{R}^{n\times n}$ to denote the $n$ dimensional identity matrix. 
%
$ B_{\alpha}^{n}:=\{x \in \mathbb{R}^{n}:  \|x\|:= \|x\|_2
< \alpha \}$ is the open 
$\alpha$-ball in $\mathbb{R}^n$. Similarly, $\mathcal{S}^{ n-1}:=\{x \in \mathbb{R}^{n}:\|x\|_2=1\}$ is the unit sphere in $\mathbb{R}^{n}$ and $\mathbb{T}$ denotes unit circle in complex plane. $\rho(A)$,  $\|A\|_2$, $\|A\|_{F}$, $det(A)$, $tr(A)$ and $\sigma(A)$   represent the spectral radius,  matrix 2-norm, Frobenius norm, determinant, trace and set of eigenvalues(spectrum) of $A$, respectively. 
When, subscript under norm is not specified, automatically matrix 2-norm is assumed.
For a positive definite matrix $A$, largest and smallest eigenvaues are denoted by $\lambda_{max}(A)$ and $\lambda_{min}(A)$, respectively. Associated with every rectangular matrix $F \in \mathbb{R}^{N \times n }$ are its' singular values $\sigma_{1}(F) \geq \sigma_{2}(F), \ldots, \sigma_{n}(F) \geq 0$, where without loss of generality we assume that $N>n$. A variational characterization of each  singular values $\sigma_{k}(F)$ follows from Courant-Fischer:
\begin{align}
 \nonumber \sigma_{k}(F)&=\max_{V \subset \mathbb{R}^{n}: dim(V)=k}\hspace{5pt} \min_{ x \in V \cap \mathcal{S}^{n-1} }\big\|F x\big\| \\ & \nonumber =
    \min_{V \subset \mathbb{R}^{n}: dim(V)=n-k+1} \hspace{5pt} \max_{ x \in V \cap \mathcal{S}^{n-1}}\big\|F x\big\|
\end{align}

 A function $g: \mathbb{R}^{n} \rightarrow \mathbb{R}^{p}$ is Lipschitz with constant $L$ if for every $x,y \in \mathbb{R}^{n} $, $\|g(x)-g(y)\| \leq L\|x-y\|$.  Notations like $O,\Theta$ and $\Omega$ will be used to highlight the dependence (\emph{non-asymptotically}) w.r.t number of iterations $N$ and dimensionality of the state space $n$. If a statement is only asymptotically true we will highlight it separately.

Space of probability measure on  $\mathcal{X}$(continuous space) is denoted by  $\mathcal{P(\mathcal{X})}$ and space of its Borel subsets is represented by $\mathbb{B}\big(\mathcal{P(\mathcal{X})}\big)$.
On a metric space $(\mathcal{X},d)$, for $\mu, \nu \in \mathcal{P(\mathcal{X})}$, we define Wasserstein metric of order $p \in [1, \infty)$~as
\begin{equation}
\label{eq:WM}
    W_{p}^{d} (\nu,\mu)= \bigg(\inf_{(X,Y) \in \Gamma(\nu,\mu)} \mathbb{E}~d^{p}(X,Y)\bigg)^{\frac{1}{p}};
\end{equation}
here, $\Gamma(\nu,\mu) \in P(\mathcal{X}^{2})$, and $(X,Y) \in \Gamma(\nu,\mu)$ implies that random variables $(X,Y)$ follow some probability distributions on $P(\mathcal{X}^{2})$ with marginals $\nu$ and $\mu$. Another way of comparing two probability distributions on $\mathcal{X}$ is via relative entropy, which is defined as
\begin{equation}
\label{eq:ent} 
    H(\nu|| \mu)=\left\{ \begin{array}{lr}
    \int \log\bigg(\frac{d\nu}{d\mu}\bigg) d\nu, & \text{if}~ \nu << \mu,
         \\ +\infty, & \text{otherwise}. 
         \end{array}\right.
\end{equation}

\section{Finite Dimensional Linear Transformations: Basis-free approach}
\label{sec:FDL-Base-free}
Given two finite dimensional vector spaces $\mathcal{V}$ and $\mathcal{W}$ with field $\mathbb{F}$ (can be $\mathbb{R}$ or $\mathbb{C}$ depending on the specific problem), $A \in L(\mathcal{V},\mathcal{W})$ means that $A$ is linear transformation from $\mathcal{V}$ to $\mathcal{W}$. Given set of vectors $[v_{i}]_{i=1}^{k} \in \mathcal{V}$ , we say that $ v \in $span$[v_{1}, v_{2}, \ldots, v_{m}]$ if there exists constants $[a_{i}]_{i=1}^{k}$ such that:
\begin{equation}
    v=a_{1}v_1+a_{2}v_{2}+\ldots+a_{k}v_{k}, \hspace{5pt}, \textit{where each} \hspace{5pt} a_{i} \in \mathbb{F}.
\end{equation}
As in \cite{dullerud2013course} chapter 1, on a finite dimensional vector space $\mathcal{V}$, we define its \emph{dimension}, denoted by $dim(\mathcal{V})$, as the smallest number $n$, such that there exists vectors $[v_{i}]_{i=1}^{n}$ such that:
\begin{equation}
    span[v_1 ,v_2, \ldots,v_{n}]=\mathcal{V}.
\end{equation}
If this is true, then $[v_{i}]_{i=1}^{n}$ is a \emph{basis} for $\mathcal{V}$, and automatically satisfy \emph{linear independence}:
\begin{equation}
    a_{1}v_{1}+a_{2}v_{2}+ \ldots+ a_{n}v_{n}=0,
\end{equation}
then $a_{i}=0$ for all $i \in [1,2,\ldots,n]$. If $\mathcal{V}$ and $\mathcal{W}$ are finite dimensional vector space with complex field, they can be can be equipped with an \emph{inner product structure} i.e.,
\begin{align}
    & \nonumber \|v\|_{\mathcal{V}}:= \langle v,v \rangle_{\mathcal{V}} \geq 0, \hspace{3pt} \textit{for all}, \hspace{3pt} v \in \mathcal{V} \hspace{3pt} \textit{and 0 $\iff v=0$ }  \\ & \nonumber \langle x_1+x_2,y \rangle_{\mathcal{V}}= \langle x_{1},y\rangle_{\mathcal{V}}+\langle x_{2},y\rangle_{\mathcal{V}}. \\ & \nonumber \langle x,\alpha y\rangle= \overline{\alpha}\langle x,y\rangle_{\mathcal{V}}, \hspace{5pt} \textit{where $\overline{\alpha}$ means complex conjugate of $\alpha$ } \\ & \label{eq:innproduct} \overline{\langle x,y\rangle_{\mathcal{V}}}=\langle y,x \rangle_{\mathcal{V}}.
\end{align}


After defining a similar inner product structure on $\mathcal{W}$, one can define an \emph{adjoint transformation} $A^{*} \in L(\mathcal{W},\mathcal{V})$ such that.
\begin{align}
     & \label{eq:adjinp}\langle Av ,w \rangle_{\mathcal{W}}= \langle v, A^{*}w \rangle_{\mathcal{V}}, \\ & \|A\|_{L(\mathcal{V},\mathcal{W})}:=\sup_{\|v\|_{\mathcal{V}}=1} \|Av\|_{\mathcal{W}}=\|A^{*}\|_{L(\mathcal{W},\mathcal{V})},  \\ & \|AA^{*}\|_{L(\mathcal{W},\mathcal{W})}=\|A^{*}A\|_{L(\mathcal{V},\mathcal{V})}=\|A\|^{2}_{L(\mathcal{V},\mathcal{W})},
\end{align}
see e.g., chapter 2 in \cite{reed1980functional}.
Moreover, $A$ can be viewed as the following map(we will discuss the formal procedure in Section \ref{sec:invsub}):
\begin{align}
    A:\underbrace{N(A) \oplus N(A)^{\perp}}_{\mathcal{V}} \rightarrow \underbrace{Im(A) \oplus Im(A)^{\perp}}_{\mathcal{W}},
\end{align}
where $A:N(A)^{\perp} \rightarrow Im(A)=\{w \in \mathcal{W}: \exists v \in \mathcal{V}: Av=w\}$ is bijective and $N(A):=\{ v \in \mathcal{V}: Av=0 \}$ is the null space of A. Similarly, $A^{*}:Im(A) \rightarrow N(A)^{\perp}$ is bijective and $N(A^{*})=Im(A)^{\perp}$.
\begin{definition}
We say a linear transformation $A$, from vector space $\mathcal{V}$ to $\mathcal{W}$ is \emph{full rank} if $Im(A)=\mathcal{W}$.    
\end{definition}
 This brings us to a very important observation that will help us in improving learning rate for state transition matrices via OLS:
\begin{lemma}
    \label{lm:dimred} Given a \emph{full rank} $A \in L(\mathcal{V}, \mathcal{W})$ , where $1 \leq dim[\mathcal{W}]=n'< dim[\mathcal{V}]=n$. Then:
\begin{enumerate}
    \item $\big(AA^{*}\big)^{-1} \in L(\mathcal{W},\mathcal{W})$,
    \item $dim[N(A)^{\perp}]=n'$  \hspace{5pt}, $dim[N(A)]=n-n'$
    \item \label{clm:skeptic}$A^{*}\big(AA^{*}\big)^{-1} \in L(\mathcal{W},N(A)^{\perp})$, is a bijection and given $N(A)$ one can construct a matrix version of $A^{*}\big(AA^{*}\big)^{-1}$  by padding $n-n'$ rows of zeros to an $n' \times n'$ matrix   
\end{enumerate}
\end{lemma}
\begin{proof}
    First two results are obvious from preceding discussion. Let $[v_{1},v_{2}, \ldots,v_{(n-n')}]$ be linearly independent vectors that span $N(A)$. We know that image space of $A^{*}\big(AA^{*}\big)^{-1}$ should be orthogonal to $N(A)$ , so we can come up with vectors $[a_{1},a_{2}, \ldots,a_{n'}]$ whose span is $Im[A^{*}\big(AA^{*}\big)^{-1}]$ with only restriction of orthogonality w.r.t $v_{i}$'s i.e. for each $j \in [1,2,\ldots,n']$ 
\begin{align}
    \label{eq:paddmatrix}
    \langle v_{i},a_{j} \rangle=0 \hspace{5pt} \forall i \in [1, n-n']. 
\end{align}    
Let $a_{j}(n'+1)= \ldots a_{j}(n)=0$, then we have $n'$  equations in \eqref{eq:paddmatrix} with $n'$ unknowns and the claim follows.  
\end{proof}
\section{Spectral Theorem for non-Hermitian Linear Transformations via orthogonal projections onto its' invariant sub spaces}
\label{sec:invsub}
\begin{equation}
\label{eq:LGS}
    x_{t+1}= Ax_t+ w_{t}, \hspace{10pt} \rho(A) <1 \hspace{10pt} \text{and i.i.d }~ w_{t} \thicksim \mathcal{N}(0,\mathcal{I}_n).
\end{equation}
It mixes to stationary distribution $\mu_{\infty} \thicksim \mathcal{N}(0, P_{\infty})$, where the controllability grammian $P_{\infty}$ is the unique positive definite solution of the following Lyapunov equation:
\begin{equation}
\label{eq:contgram}
    A^{*}P_{\infty}A-P_{\infty}+I_{n}=0.
\end{equation}


Position or magnitude of eigenvalues associated to a linear operator $A$ only provides partial information about its' properties (for the ease of exposition, throughout this paper we will assume that $A$ does not have any non-trivial null space). In this paper, we will study operators and matrices via their actions on associated invariant subspaces and concepts like generalized eigenevectors. Topic in itself can take a decent amount of work and we refer the reader to \cite{axler1995down, axler1997linear}. However, we will try here to give the reader a quick intuition of this approach, with minimal assumptions as general case will follow from minor adjustments. One of the advantage of taking this approach is: \emph{control on the $k-th$ power of matrix norm, independent of the basis structure .}   Roughly speaking, algebraic multiplicity of eigenvalues follow from chatacteristic polynomial of the matirx. 
\begin{equation}
    \label{eq:detcharpoly} det(zI-A)= \prod_{i=1}^{K} (z-\lambda_{i})^{m_i},
\end{equation}
where $\lambda_{i}$ are distinct with multiplicity $m_{i}$ and  $\sum_{i=1}^{K} m_i=n$. Since algebraic multiplicity of eigenvalue $\lambda_{i}$ is $m_{i}$ , we denote it by $AM(\lambda_{i})=m_{i}$. Similarly with each $\lambda_{i} \in \sigma(A)$, their is an associated set of eigenvectors and dimension of their span corresponds to geometric multiplicity of $\lambda_{i}$ , which we denote by $GM(\lambda_{i})=dim[N(A-\lambda_{i}I)]$. Recall, from linear algebra:
\begin{lemma}
    Let $A \in \mathbb{C}^{n \times n}$. Then $A$ is diagonalizable iff there is a set of $n$ linearly independent vectors, each of which is an eigenvector of $A$. 
\end{lemma}


So, in a situation where $GM(\lambda_{i}) <AM(\lambda_{i}) $, eigenvectors do not span $\mathbb{C}^{n}$ and one resorts with spanning the underlying state space by direct sum decomposition of $A-$ invariant subspaces (which might be spanned by more that one linearly independent vector, comprising of eigenvector and generalized eigenvectors). 
\begin{definition}
    Given a matrix $A \in \mathbb{C}^{n \times n}$ and a subspace $M \subset \mathbb{C}^{n}$, we say that $M$ is an $A-$ invariant subspace if $AM \subset M$.
\end{definition}

Indeed, preceding philosophy is underlying principal of Jordan canonical forms, but its' geometric intricacies can not be ignored when dealing with High Dimensional estimation and control problems.
\begin{proposition}
\label{prop:dirsuminv}
We can decompose the underlying state space as a direct sum decomposition of $A-$ invariant subspaces $[M_{\lambda_{i}}]_{i=1}^{K}$ denoted by:  
\begin{equation}
\label{eq:directsumAinv}
    \mathbb{C}^{n}= M_{\lambda_{1}} \oplus M_{\lambda_{2}} \oplus \ldots \oplus M_{\lambda_{K}},
\end{equation}    
where $M_{\lambda_{i}}=N(A-\lambda_{i}I)^{m_{i}}$ is called the \textbf{Generalized eigenspace} associated with eiegenvalue $\lambda_{i}$(see e.g., theorem 3.11 in \cite{axler1995down}).
\end{proposition}
\paragraph{Direct sum decomposition of the state space via projections onto $A-$ invariant subspace}
Furthermore, one can define orthogonal projection matrices $[P_{\lambda_{i}}]_{i=1} ^{K}$ associated to these invariant subspaces
and the identity matrix can be written as:
\begin{equation}
    \label{eq:addId} I_{n}=P_{\lambda_{1}} \oplus P_{\lambda_2} \oplus \ldots \oplus P_{\lambda_{K}}.
\end{equation}
Consequently, every $x \in \mathbb{R}^{n}$ can be written uniquely as $\oplus_{i=1}^{K} x_{\lambda_i}$, where $x_{\lambda_i}$ is an element of subspace $M_{\lambda_i}$. How ? just for the intuition assume we are given a $k$ dimensional subspace $\mathcal{V} \subset \mathbb{R}^{n}$ with its basis vectors $(v_1,v_2, \ldots, v_{k})$, then orthogonal projection of an arbitray vector $x \in \mathbb{R}^{n}$ onto $\mathcal{V}$ is uniquely represented by:
\begin{equation}
    P_{\mathcal{V}}(x)=a_{1}(x) v_1+ a_{2}(x) v_2+ \ldots a_{k}(x) v_k,
\end{equation}
where $[a_{i}(x)]_{i=1}^{k}$ are minimizers of the following optimization problem: 
\begin{equation}
    \min_{a_1,a_2, \ldots,a_{k}}\|x-(a_1 v_1 +a_2 v_2+ \ldots a_{k} v_{k})\|.
\end{equation}  
Let $V:=[v_1, v_2, \ldots, v_k]$ that is $[v_{i}]_{i=1}^{k}$ be the columns of matrix $V$, then the matrix version of orthogonal projection of a given vector $x$ on subspace $\mathcal{V}$ and associated coefficients are: 
\begin{align}
 & \nonumber  P_{\mathcal{V}}(x)=V(V^{*} V)^{-1}V^{*}x \\ & \textit{and} \hspace{5pt} \label{eq:projcoeff} a(x)=(V^{*} V)^{-1}V^{*}x, 
\end{align}
respectively. 

Now we can re-arrange the basis according to \eqref{eq:directsumAinv} via projection operators and have the following diagonal blocks(corresponding to $A-$ invariant sub-spaces) of matrices:
\[
  \hat{A} =
  \begin{bmatrix}
    A_{\lambda_{1}} & & \\
    & \ddots & \\
    & & A_{\lambda_{k}},
  \end{bmatrix}
\]
where superscript was used to signify this representation corresponds to change of basis according to proposition \ref{prop:dirsuminv}. Consider the following instructive example: \newline

\textbf{Example 1.} \label{ex:ex1} Assume that given $A \in \mathbb{C}^{n \times n }$, comprises of only two distinct eigenvalues: $\lambda$ and $\rho$ , such that $AM(\lambda)=n_1$ and $GM(\lambda)=1$. Similarly, $AM(\rho)=n_2$ (also $n_1+n_2=n$) and $GM(\rho)=1$. Given an eigenvector $v_1$  such that $Av_1=\lambda v_1$, we generate generalized eigenvector $v_2, v_3, \ldots, v_{n_1}$ recursively as $(A-\lambda I)v_2=v_1$, $(A-\lambda I)v_3=v_2$ and so on up to $(A-\lambda I)v_{n_1}=v_{n_1-1}$. We also have the following $k$ -th step iteration:
\begin{align}
     & \label{eq:recursiveJordan} A^{k}v_1 =\lambda^{k}v_1 \\ & \nonumber A^{k}v_2=\lambda^{k}v_{2} + \binom{k}{1}\lambda^{k-1}v_1 \\ & \nonumber A^{k}v_3=\lambda^{k}v_{3} + \binom{k}{1}\lambda^{k-1}v_2 +  \binom{k}{2} \lambda^{k-2}v_1 \\ & \nonumber \hspace{9pt} \ldots= \hspace{9pt} \ldots \\ & \nonumber
    A^{k}v_{n_1}=\lambda^{k}v_{n_1}+ \binom{k}{1}\lambda^{k-1}v_{n_1-1}+\ldots \binom{k}{n_1-1}v_1.  
\end{align}
Similarly given $A {w_1}= \rho w_1$, we recursively generate generalized eigenvectors up to $(A-\rho I) w_{n_2}=w_{n_2 -1} $.
Now under new basis: $[v_{1}, \ldots, v_{n_{1}-1}, v_{n_1}, w_1, \ldots, w_{n_2}]$, $A$ can be represent as:
    \[
  \hat{A} =
  \begin{bmatrix}
    A_{\lambda} &  \\
     & A_{\rho}
  \end{bmatrix}
  \]
$A_{\lambda}=\lambda I_{n_1}-N_{n_1}$ and $A_{\rho}=\rho I_{n_2}-N_{n_2}$  Where, $N_{n_1}$ and $N_{n_2}$ are $n_1$ and $n_2$ dimensional Nilpotent matrices, respectively. Two $A-$ invariant subspace are $M_{\lambda}=span[v_1,\ldots, v_{n_1}]$ and $M_{\rho}=span[w_1, \ldots, w_{n_2}].$

\section{Controlling the norm of $k-$ th power of Linear Transformation}
\label{sec:kpow}
\begin{assumption}
We assume here that each distinct eigenvalue $\lambda_{i}$ has geometric multiplicity of $1$ and $m_{i}>1$. This is merely for the ease of exposition; as we can always replace a Jordan block by a diagonal element.  
\end{assumption}

\begin{theorem} Upper bound on the norm of the $k-th$ iteration associated to action of matrix $A$ on invariant subspace $M_{\lambda_{i}}$, with block size $m_{i}$ is precisely given as:  
    \begin{align}
    \label{eq:ubdexact}
\|A^{k}_{\lambda_{i}}\|_{2} \leq |\lambda_{i}|^{k} k^{m_{i}-1}\sum_{m=0}^{m_{i}-1} \frac{1}{|\lambda_{i}|^{m}}, 
    \end{align}
Moreover, if $|\lambda_{i}| \in (0,1)$ then:
\begin{align}
\label{eq:best-nonasym}
        \|A^{k}_{\lambda_{i}}\|_{2} \leq k^{m_{i}-1}
          |\lambda_{i}|^{k+1- m_{i} } 
\end{align}

and for any $k \in \mathbb{N}$ such that
\begin{equation}
    \label{eq:ycln}
    k > \frac{\ln(m_{i})}{\ln\big(\frac{1}{|\lambda_{i}|}\big)} + \frac{[m_{i}-1] \ln (k)}{\ln\big(\frac{1}{|\lambda_{i}|}\big)} + (m_{i}-1)
\end{equation}
we have that  $\|A^{k}_{\lambda_{i}}\|_{2} <1$.
\end{theorem}
\begin{proof}
\begin{align}
    \label{eq:normtpltz} \|A^{k}_{\lambda_{i}}\|_{2} &=\|(\lambda_{i} I_{m_{i}}+N_{m_{i}} )^{k} \|=\|\sum_{m=0}^{k} \binom{k}{m} N_{m_{i}}^{m} \lambda_{i}^{k-m }\|    \\ & \nonumber \leq  \sum_{m=0}^{m_{i}-1} \binom{k}{m} |\lambda_{i}|^{k-m}  \leq |\lambda_{i}|^{k} k^{m_{i}-1} \sum_{m=0}^{m_{i}-1} \frac{1}{|\lambda_{i}|^{m}}. 
\end{align}
   For the first two equalities and first inequality we have used the fact that any Jordan block of size $m_{i}$ can be written as $(\lambda I_{m_{i}}+N_{m_{i}} )$, where $N_{m_{i}}$ is a Nilpotent matrix i.e., for any $m \geq m_{i}$, $N_{m_{i}}^{m}=0$ and for $m<m_{i}, \|N_{m_{i}}^{m}\|=1$.  In the last inequality we have used the fact  $\binom{k}{m} \leq k^{m_{i}-1}$ for $m \in [0, \ldots, m_{i}-1 ]$. Given $|\lambda_{i}| \in (0,1)$, $ \sum_{m=0}^{m_{i}-1} \frac{1}{|\lambda_{i}|^{m}} \leq m_{i} |\lambda_{i}|^{1-m_{i}}$ and the results follows.
\end{proof}
Although our upper bound on the $\ell_{2}$ norm and consequently lower bound for $k$ that suffices for $\|A_{\lambda_{i}}^{k}\|_{2}<1$ is non-trivial, but using finite sum of geometric series, can be improved to
\begin{align}
    & \label{eq:btrubdstbl} \|A_{\lambda_{i}}^{k}\|_{2} \leq k^{m_{i}-1}|\lambda_{i}|^{k} \bigg(\frac{1-|\lambda_i|}{1-|\lambda_{i}|^{m_i}}\bigg), \hspace{10pt} \textit{and} \\ &  \hspace{5pt} k> \frac{[m_{i}-1] \ln (k)}{\ln\big(\frac{1}{|\lambda_{i}|}\big)}+\frac{\ln{\bigg[ \frac{1- |\lambda_{i}|^{m_i}}{1-|\lambda_{i}|} \bigg]}}{\ln{\big( \frac{1}{|\lambda_{i}|}\big)}}, \hspace{5pt} \textit{implies} \hspace{5pt} \|A_{\lambda_{i}}^{k}\|_{2}<1.
\end{align}
 Taylors' expansion can be used to show
 \begin{equation}
    \ln{\bigg[ \frac{1- |\lambda_{i}|^{m_i}}{1-|\lambda_{i}|} \bigg]} \leq \bigg(\frac{1}{m_{i}-1}\bigg) \ln{\bigg(\frac{1}{1-|\lambda|^{m_i}}\bigg) +\frac{|\lambda_{i}|}{1-|\lambda_{i}|} },
\end{equation}
and get a tighter lower bound on sufficient condition for $k$ such that $\|A_{\lambda_{i}}^{k}\|<1$.
\begin{theorem}
\label{thm:A-sp-K-l}
    Given a full-rank stable matrix $A \in \mathbb{C}^{n \times n}$, with distinct eigenvalues $[\lambda_{i}]_{i=1}^{K}$. Assume that geometric multiplicity of each distinct eigenvalue is 1 while algebraic multiplicity is $AM(\lambda_{i})=m_{i}$, then for any $k$ greater than 
\begin{align}
    \nonumber \hat{k}= \min \big[k \in \mathbb{N} :  k & \geq \max_{i \in 1, \ldots, K} \bigg(\frac{4[m_{i}-1] \ln m_{i} }{\ln \frac{1}{|\lambda_{i}|}} \bigg) \big], \\ & \nonumber  \hspace{5pt} \|A^{k}\|_{2} <1 
\end{align} 
\end{theorem}
\begin{proof}
Since, generalized eigenvectors corresponding to distinct eigenvalues are linearly independent(see e.g., 8.13 in \cite{axler1997linear}), $[P_{\lambda_{i}}]_{i=1}^{K}$ are orthogonal projections. Consequently,
\begin{align}
   &\nonumber \|A^{k}{x} \|_{2}= \bigg \|A^{k} \bigg[\sum_{i=1}^{K}P_{\lambda_{i}} x\bigg] \bigg \|=   \bigg\|\sum_{i=1}^{K} A^{k}_{\lambda_{i}} x_{\lambda_{i}} \bigg \| \\ & \nonumber= \sqrt{ \sum_{i=1}^{K} \bigg[  \langle A_{\lambda_{i}}^{k}x_{\lambda_{i}}, A_{\lambda_{i}}^{k}x_{\lambda_{i}} \rangle +2 \sum_{j>i}^{K} \langle A_{\lambda_{i}}^{k}x_{\lambda_{i}} , A_{\lambda_{j}}^{k}x_{\lambda_{j}}\rangle \bigg]} \\ & \label{eq:crosstermorth} =\sqrt{ \sum_{i=1}^{K}   \langle A_{\lambda_{i}}^{k}x_{\lambda_{i}}, A_{\lambda_{i}}^{k}x_{\lambda_{i}}\rangle} \leq \sqrt{\sum_{i=1}^{K} \big\|A_{\lambda_{i}}^{k}\big\|^2\big\|x_{\lambda_{i}}\big\|^{2}} \\ & \label{eq: kgeq1cont}< \sqrt{\sum_{i=1}^{K}\big\|x_{\lambda_{i}}\big\|^{2}}=\big\|x\big\|.
\end{align}
where $x_{\lambda_{i}}:=P_{\lambda_{i}}x$, \eqref{eq:crosstermorth} follows from the fact that $N(A_{\lambda_{i}}^{*})=\oplus_{j \neq i}^{K} M_{\lambda_{j}} $ and $\bigg\langle x_{\lambda_{i}}, (A_{\lambda_{i}}^{k})^{*}A_{\lambda_{i}}^{k} x_{\lambda_{i}}\bigg\rangle \leq \big\|A_{\lambda_{i}}^{k}\big\|^2 \big\|x_{\lambda_{i}}\big\|^{2}$. Last inequality, \eqref{eq: kgeq1cont} follows from the fact that by hypothesis $k> \hat{k}$ and 
$\frac{\ln(m_{i})}{\ln\big(\frac{1}{|\lambda_{i}|}\big)} + \frac{[m_{i}-1] \ln (k)}{\ln\big(\frac{1}{|\lambda_{i}|}\big)} + (m_{i}-1)< \bigg(\frac{4[m_{i}-1] \ln m_{i} }{\ln \frac{1}{|\lambda_{i}|}} \bigg)$ for all $i \in [1,2 , \ldots,K]$.
\end{proof}

\begin{remark}
    Since for a stable state transition matrix $|\lambda_{i}| \leq \rho(A)$ and $m_{i} \leq n$ for each $i \in [1,2, \ldots,K]$, therefore $\Gamma(n, \rho)= O  \bigg(\frac{[n-1] \ln n}{\ln \frac{1}{\rho}}\bigg)$.
\end{remark}
\vspace{1pt}
As any diagonalisable linear transformation $A$ with $\rho$ on diagonals satisfy $\|A^{k}\|= \rho^{k}<1$ for all $k \in \mathbb{N}$, one might question tightness of preceding bound, which brings us to following result:
\begin{definition} [Stable with Strong Spatial Correlations, S-w-SSCs]
 Consider the following example of a stable state transition matrix with only one distinct eigenvalue of $\rho \in (0,1)$  
\begin{align}
\label{eq:S-w-SSCS}
J_{n}(\rho):=
\begin{bmatrix}
\rho & 1 & 0 &  \cdots &0 &0 \\
0 & \rho  & 1 & \ddots & 0 & 0\\
0 & 0 & \rho  & \ddots & 0 & 0 \\
0 & 0 & 0 & \ddots & 1 & 0
\\
\vdots & \ddots & \ddots & \ddots & \rho & 1 \\
0 & 0 & 0 & \cdots & 0 & \rho   
\end{bmatrix}
\end{align}  
with algebraic multiplicity of $n$ but only one linearly independent eigenvector.
\end{definition}
\begin{theorem}
\label{thm:lowerbound}
    Given any $n \in \mathbb{N}$ and spectral radius $\rho \in (0,1)$, there exists a linear transformation $A \in \mathbb{R}^{n \times n}$ with $\rho(A)=\rho$ and $\|A^{n-1}\|>1$.    
\end{theorem}
\begin{proof}

Let $A:=J_{n}(\rho)$ and $[e_{i}]_{i=1}^{n}$ be canonical basis of $\mathbb{C}^{n}$, i.e., $1$ at $i-$th position and 0 everywhere else, then notice for $x_{0}=e_{n}$: 
\begin{align}
    \nonumber A^{n-1}e_{n} &= \sum_{m=0}^{n-1} \binom{n-1}{m} \rho^{(n-1)-m}e_{n-m}\\ & =\sum_{m=0}^{n-2} \binom{n-1}{m} \rho^{(n-1)-m}e_{n-m} +e_{1}.
\end{align}
Therefore,
\begin{equation}
\label{eq:lbdjdn}
    \|A^{n-1}\|_{2}:=\sup_{x \in \mathcal{S}^{n-1}}\|A^{n-1}x\| \geq \|A^{n-1}e_{n}\|> \|e_{1}\|=1.
\end{equation}
\end{proof}
\begin{remark}
Now we are in a position to interpret consequences of bounds on the operator norm of the powers of linear transformations that we have derived after a painstaking process of general spectral decomposition using projections onto invariant subspaces of a linear transformation.  From dynamical systems point of view, in contrast with hermitian or diagonalizable setting where original dynamical system can be decomposed into $n$ one dimensional dynamical systems with no spatial correlations, for non-Hermitian case we can only `hope' to decompose the original dynamical system into multiple lower dimensional dynamical systems that are independent of each other(even in statistical sense as we will see this in subsection \ref{subsec:Sp-gl} of Section \ref{sec:exactOLS}). But, when a linear transformation has only one distinct eigenvalue and span of corresponding eigenvectors is only one dimensional as in S-w-SSCs, dynamics are stable but \emph{spatially inseparable}. Consequently, if the underlying state space is high dimensional and dynamical system follows state transitions from S-w-SSCs, there exists an initial condition on $S^{n-1}$ such that it can take a considerable amount of time for the system to be inside open unit ball and stay there forever as shown by \eqref{eq:lbdjdn}. Another vague but useful interpretation for high dimensional stable dynamical system is: strong spatial coupling between $n-$ one dimensional systems allow them to interact in a way which can lead to transient behavior of $\|A^{k}\|$(allowing it to grow arbitrarily for a noticeable period ).   Let us emphasise that property of being strongly \emph{spatially correlated} is related to \emph{large discrepancy between algebraic and geometric multiplicity of eigenvalues}, and only one Jordan block corresponding to $n$ dimensional dynamics as in \eqref{eq:S-w-SSCS} is merely a graphical representation of this phenomenon.     
\end{remark}


In the process we also discovered a quantitative handle on the decay of $\|A^{k}\|$, preciesely:
\begin{theorem}
    For a stable matrix $A$ with $K$ distinct eigenvalues, let discrepancy related to eigenvalue $\lambda_{i}$ be $D_{\lambda_{i}}:= AM(\lambda_{i})-GM(\lambda_{i})$, then 
    \begin{equation}
        \label{eq:qnthdl} \|A^{k}\| \leq \max_{1 \leq i \leq K}  k^{D_{\lambda_{i}}} |\lambda_{i}|^{k} \bigg( \frac{1-|\lambda_{i}|}{1-|\lambda_{i}|^{D_{\lambda_{i}+1}}}\bigg)
    \end{equation}
\end{theorem}

\begin{proof}
Recall that $\|P_{\lambda_{i}}(x)\|=d(x, P_{\lambda_{i}^{\perp}})$ and for all $x \in S^{n-1}$ we have that $\sum_{i=1}^{K} d^{2}(x,P_{\lambda_{i}^{\perp}})=\|x\|_{2}$. Furthermore,
\begin{align}
   & \nonumber \|A^{k}{x} \|_{2} \leq \sqrt{\sum_{i=1}^{K} \big\|A_{\lambda_{i}}^{k}\big\|^{2} \big\|P_{\lambda_{i}}(x) \big\|^{2}}= \sqrt{\sum_{i=1}^{K} \big\|A_{\lambda_{i}}^{k}\big\|^{2} d^{2}(x, P_{\lambda_{i}^{\perp}})} \\ & \label{eq:lkbtr} \leq \sqrt{\sum_{i=1}^{K} k^{2(m_i - 1)} |\lambda_{i}|^{2k} \bigg( \frac{1-|\lambda_{i}|}{1-|\lambda_{i}|^{m_i}}\bigg)^{2}d^{2}(x, P_{\lambda_{i}^{\perp}})}
\end{align}
where \eqref{eq:lkbtr} follows from using tighter upper bound from \eqref{eq:btrubdstbl}, now without loss of generality assume that 
\begin{equation}
      j=\arg \max_{1 \leq i \leq K} \bigg[ k^{(m_i-1)} |\lambda_{i}|^{k} \bigg(\frac{1-|\lambda_{i}|}{1-|\lambda_{i}|^{m_i}}\bigg)  \bigg].
\end{equation}
Then a simple convexity argument reveals:
\begin{align}
     \nonumber \sup_{x \in S^{n-1}}&  \sum_{i=1}^{K} k^{2(m_i - 1)} |\lambda_{i}|^{2k} \bigg( \frac{1-|\lambda_{i}|}{1-|\lambda_{i}|^{m_i}}\bigg)^{2}d^{2}(x, P_{\lambda_{i}^{\perp}}) \\ &  = k^{2(m_j - 1)} |\lambda_{j}|^{2k} \bigg( \frac{1-|\lambda_{j}|}{1-|\lambda_{j}|^{m_j}}\bigg)^{2},
\end{align}
where equality follows by picking any $x \in S^{n-1}$ with $d^{2}(x, P_{\lambda_{j}^{\perp}})=1$.
\end{proof}
A trivial corollary reveals an upper bound on operator norm of a finite dimensional linear transformation with all eigenvalues strictly inside unit circle
\begin{corollary}
        Given a stable state transition matrix $A$ with $K$ distinct eigenvalues $[\lambda_{i}]_{i=1}^{k}$ along with their associated discrepancy 
        $D_{\lambda_{i}}$, operator norm can be upper bounded as: 
        \begin{equation}
        \label{eq:2-norm}
          \|A\|_{2} \leq \max_{1 \leq i \leq K}  |\lambda_{i}|\bigg(\frac{ 1-|\lambda_{i}|}{1-|\lambda_{i}|^{D_{\lambda_{i}}+1}}\bigg)
        \end{equation}
\end{corollary}
\begin{remark}
Our analysis reveal that operator norm of $k$-th power of stable state transition matrix $A$ is a maximum over the operator norm of $k$-th power of lower dimensional linear transformations corresponding to restriction of $A$ to its' invariant subspaces. Furthermore, operator norm of $k$-th power of $A$ restricted to its' invariant subspace with eigenvalue $\lambda$ is at most a product of $(i)$ polynomial in $k$ with degree equal to size of the invariant subspace minus one(i.e., discrepancy of eigenvalue $\lambda$) $(ii)$ an exponential which decays with increase in $k$ and the rate is determined by magnitude of $\lambda$, modulo a constant dependent on magnitude of $\lambda$ and size of the invariant subspace. Therefore, as opposed to common belief, there is no direct causal relationship between spectral radius and decay rate of $\|A^{k}\|$ for finite powers of $A$      
\end{remark}
Our analysis is first of a kind where we translate the knowledge on entire spectrum of non-Hermitian, finite dimensional linear transformation into quantitative behavior of its powers, including evolution of norm of its' powers. One of the frequently used quantitative handle on $\|A^{k}\|$, is provided by \cite{tsiamis2021linear}. In an attempt to circumvent appearance of the condition number of similarity transform associated with Jordan blocks, they employed Schur form with   assumptions of upper bound on $\|A\|_{2}$, i.e., $\|A\|_{2} \leq M$ and all the eigenvalues of $A$ are on or inside the unit circle in complex plane, then quantitative handle is precisely $\|A^{k}\|_{2} \leq (ek)^{n-1}\max(M^{n},1)$. To begin with, our analysis shows operator norm is independent of basis choice and $\|A\|_{2} \leq \max_{1 \leq i \leq K}  |\lambda_{i}|\bigg(\frac{ 1-|\lambda_{i}|}{1-|\lambda_{i}|^{D_{\lambda_{i}}+1}}\bigg)$ when eigenvalues are strictly inside the unit circle. The worst case for eigenvalue on unit circle will be discrepancy of $n-1$ and can be bounded  simply as $  \big\|A^{k}\big\|=\big\|(\lambda I+N)^{k}\big\|=\bigg\|\sum_{m=0}^{k} \binom{k}{m} N^{m} \lambda^{k-m}\bigg\| \leq \sum_{m=0}^{k\land (n-1)} \binom{k}{m}|\lambda|^{k-m}  = \sum_{m=0}^{k} \binom{k}{m}=2^{k}$ for $k<n$. A general purpose upper bound for $k \geq n$ is $ \|A^{k}\| \leq \sum_{m=0}^{n-1} \binom{k}{m}=2^{k}-\sum_{m=n}^{k} \binom{k}{m}$, playing around with binomial expansions/ stirling approximation and picking $k$ as a function of $n$ can often lead to good quantitative results; if we chose  $k=2n$ we get $\|A^{2n}\| \leq 2^{n}-n$.
Apart from analytical gaps, major concern with previously existing result is non-existent geometric interpretation,  consider the following question:
\begin{remark}
    Assume that via some oracle you are provided with exact value of $M$, what can you say about statistical properties of data generated with state transition $A$? An overwhelming answer from related works in literature would be this growth rate will dictate temporal dependence between covariates.
    Now if instead you are given discrepancy of distinct eigenvalues and magnitude via some oracle what information do you possess? (Hint: independent lower dimensional dynamical systems and spatial independence between multiple blocks of the data matrix. As we will see in the coming section and future that this realization is a giant leap for Random matrices with correlations in space and time ) 
\end{remark}            
Existing results in finite time system identification either assume as a \emph{priori} norm bound on $\|A\|$ and based on the magnitude of its' eigenvalues conclude about norm of its powers, see e.g., Lemma 1 in \cite{tsiamis2021linear}. In fact our analysis is amenable to unstable and marginally stable case but then we will have to use upper bound \eqref{eq:ubdexact}.
Given size of the Jordan blocks and eigenvalue that populates it, one has entire knowledge of spectrum of $A$ and as shown in our proof norm of the \emph{similarity transforms} is not required, so estimates in \cite{simchowitz2018learning} have redundancies. 

\section{Concentration of measure phenomenon and Littlewood-Offord problem }
\label{sec:tensorization}

Let $(x_{i})_{i=1}^{N}$ be centered and bounded size on average  i.e., variance of $O(1)$, then one would expect $S_{N}:=\sum_{i=1}^{N}x_{i}$ to vary in an interval of $O(N)$. However, \emph{under sufficient independence assumption} between each individual components $x_1,x_2, \ldots,x_{N}$, the sum concentrates in a much narrower interval of size $O(\sqrt{N})$ i.e., $P(|S_{N}|\geq \delta \sqrt{N}) = O(\frac{1}{\delta^{2}})$. This is because each individual variable to collectively vary in a way to produce deviation of $O(N)$ becomes more and more less likely with increase in variables and this remarkable phenomenon is called \emph{concentration of measure} 

In fact, trajectory from \emph{one-dimensional stable ARMA} model, with $x_{0}=0$:
\begin{equation}
\label{eq:1dimarma}
    x_{i+1}=\lambda x_{i}+w_{i}, \hspace{3pt} w_{i} \thicksim N(0,1) 
\end{equation}
for some $|\lambda|<1$, $w_{i}$ and $w_{j}$ are independent for $i \neq j$, also satisfies this phenomenon, precisely:
\begin{equation}
    P \bigg(|S_{N}|\geq \delta \sqrt{\frac{N-\sum_{i=1}^{N}|\lambda|^{2i}}{1-|\lambda|^2}}\bigg) = O\ \bigg(\frac{1}{\delta^{2}}\bigg)
\end{equation}
This phenomenon is not just limited to deviations of sums and can be extended to smooth functions of `sufficiently independent' random variables via Talagrands' Inequality: 
\begin{definition}
    We say that the probability measure $\mu$ satisfies the $L_{p}$-transportation cost inequality on $(\mathcal{X}, d)$ if there is some constant $C > 0$ such that for any probability measure $\nu$
\begin{equation}
    \mathcal{W}_{p}^{d} (\nu,\mu) \leq \sqrt{2C H(\nu|| \mu)},
\end{equation}    
and we use $\mu \in T_{p}^{d}(C)$ as a shorthand notation.
\end{definition}
Since for $p>1$, $T_{p}^{d}(C)$ implies $T_{1}^{d}(C)$ which is related to concentration of Lipschitz functions:
\begin{remark}
\label{rm: lipiid}
[Theorem 1.1 in \cite{djellout2004transportation}: $T_{1}^{d}(C)$ implies Lipschitz function concentration  ] Given any Lipschitz function $f$ on random variable $x$ with underlying distribution $\mu \in T_{1}^{d}(C)$, we have: 
\begin{equation}
    \label{eq:lipiid} P \Bigg[ \bigg| f(x) -<f>_{\mu}\bigg| > \epsilon \Bigg] \leq 2\exp\bigg(-\frac{ \epsilon ^2}{2C \|f\|_{L(d)}^2}\bigg).
\end{equation}
\end{remark}
As we will be interested in deviations and typical behaviors of random processes, like trajectories following some Markovian dynamics e.t.c, it is important to understand how it varies with number of iterations and dimension of the underlying space, a useful result is:
\begin{theorem}
\label{thm:dim_ind_tal}
[Theorem 1.1 in \cite{talagrand1996transportation}:Dimension independent tensorization of Gaussian Measure] 
Standard normal on any finite dimensional conventional metric space $\mathbb{R}^{n}$ satisfies $T_{2}(1)$. Moreover, for an $\ell_{2}$ additive metric 
\begin{equation}
    d_{(N)} ^2 (x^N,y^N):= \sqrt{\sum\nolimits_{i=1}^{N} d^2(x_i,y_i)},
\end{equation}
on product space $\mathbb{R}^{n^{\otimes N}}$, isotropic Gaussian satsfies $T_{1}^{d_{(N)}^2} \big(1\big)$.
\end{theorem}

From now on throughout the paper we will use $\ell_2$ metric and $d=d_{(N)} ^2 (x^N,y^N)$. A remarkable advantage of preceding result combined with Lipschitz function concentration result is that given a process $x=(x_1,x_2, \ldots,x_{N}) \thicksim \mu_{N}= N\bigg(0,\Sigma_{N}\bigg)$, then $\mu_{N} \in T_{1} \bigg(  \big \|\Sigma_{N}^{\frac{1}{2}} \big \|^{2} \bigg)$ and even though process might have temporal dependencies but as long as one can prove $\big\|\Sigma_{N}^{\frac{1}{2}}\big\|^{2}=O(1)$, dimension independent tensorization follows. This brings us to \emph{norm-stable} ARMA models

\begin{theorem}
\label{thm:ten_tal}[Proposition 4.1 
 in \cite{blower2005concentration}:Dimension-Independent Tensorization of Talagrands' Inequality for norm-stable ARMA models]
Consider the following model:
\begin{align}
    x_{t+1}= Ax_t+ w_{t}, \hspace{10pt} \text{and i.i.d }~ w_{t} \thicksim \mathcal{N}(0,\mathcal{I}_n).
\end{align}  
If system transition matrix, $A$ satisfies $\|A\|<1$, then the process level law of $(x_1,x_2, \ldots,x_{N})$ satisfies $T_{1}\bigg(\frac{1}{(1-\|A\|)^2}\bigg)$, if  $\|A\|=1$ we have $T_{1} \bigg(N(N+1)\bigg)$  and for $\|A\|>1$, $T_{1}\bigg(\|A\|^{N} N\bigg)$ . 
\end{theorem}
A curious reader might wonder what about the case $\rho(A)<1$, unfortunately dimension independent tensorization is not guaranteed. However, one can find a sub-trajectory that satisfies dimension independent Talagrands' inequality, see e.g., \cite{naeem2023learning}. 
A fundamental limitation of existing results, stem from the fact that they base their analysis only on the extreme singular values of the data matrix which do not offer much of geometric insight. This brings us to an extremely important, although simple to prove equality that relates the distance between different rows (whose concentration we understand well because of preceding theorem \ref{thm:ten_tal}) of the data matrix to all of its' singular values. 
\begin{theorem}
[Lemma A.4 in \cite{tao2010random} Negative second moment  \label{thm:neg_2nd_moment_ineq}] Let $1 \leq d \leq p$ and $Y \in \mathbb{R}^{d \times p}$ be a full rank matrix with singular values, $\sigma_{1}(Y) \geq \sigma_{2}(Y) \ldots \geq \sigma_{d}(Y) $. Let $v_{j}$ be the hyperplane generated by all the rows of $Y$-except the $j-th$ : i.e., span of $y_1,y_2, \ldots, y_{j-1}, y_{j+1}, \ldots, y_{d}$  for $1 \leq j \leq d$, $(e_{j})_{j=1}^{d}$ be the canonical basis of $\mathbb{R}^{d}$, then: 
 \begin{equation}
 \label{eq:neg_2nd_moment_ineq}
     \sum_{j=1}^{d} \sigma_{j}^{-2}(Y)=\sum_{ j=1}^{d}\big \langle \big(YY^{*}\big)^{-1}e_{j} ,e_{j} \big \rangle =\sum_{j=1}^{d} d_{j}^{-2}, 
 \end{equation}
where  $d_{j}^{-2}:=d^{-2}(y_j, v_j)$, distance between $y_{j}$ and the point closest to it in the subspace $v_{j}$
\end{theorem}
An important observation about orthogonal projections which will be extremely useful in providing partial information on inverse sample covariance matrix and consequently, pseudo-inverse of the data matrix in system identification problem is:
\begin{corollary}
\label{cor:vimpescpnegmomstuck}
Let $x_{j} \in \mathcal{S}^{p-1}$ such that it is orthogonal to subspace $v_{j}$ and $P_{v_{j}^{\perp}}$ be the orthogonal projection onto subspace orthogonal to $v_{j}$ then using properties of projections and cauchy-schwarz:
\begin{equation}
    \big|\langle y_{j},x_{j}\rangle\big|=\big| \langle P_{v_{j}^{\perp}}(y_j),x_{j} \rangle \big| \leq \big \| P_{v_{j}^{\perp}}(y_j) \big\| =d(y_{j},v_{j}).
\end{equation}
\end{corollary}
Our novel analysis on OLS estimation error will boil
down to studying sum of $N$ identically distributed independent variables weighted by columns of pseudoinverse of the data matrix. It turns out that this problem is related to the well known
\newline
\textbf{Littlewood-Offord Problem:} while working on random polynomials, Littlewood and Offord posed the following question: Let $a:=[a_1, . . . , a_N]$ be nonzero integers, and consider the function $f(x_1, \ldots , x_N) = \langle x,a\rangle$ How many solutions can $f (x)=r$ have with $x_{i} \in \{+1,-1\}^{N}$? which amounts to studying probability of maximal atom w.r.t non-zero constants $a$, defined as
\begin{equation}
    p(a):=\max_{r \in \mathbb{R}} P\big(f(x)=r\big).
\end{equation}
and turns out to depend heavily on the structure of co-efficients $a$ (analogously for Remark \ref{rmk:inp}, on coefficients of pseudoinverse of the data matrix, $c_{k}$)
\begin{enumerate}
    \item If $a=(1,1,\ldots,1)$ and $N$ is even:
            \begin{equation}
                p(a)= \binom{N}{\frac{N}{2}}2^{-N}=O\bigg(\frac{1}{\sqrt{N}}\bigg),
            \end{equation}
          which is essentially probability of equal heads and tails in $N$ trials.
    \item If $a=(2^{0},2^{1},2^{2},\ldots,2^{N-1})$
            \begin{equation}
                p(a)=\frac{1}{2^{N}}
            \end{equation}
    \item If $a=(1,2,\ldots,N)$
            \begin{equation}
                p(a)=\Theta \bigg(\frac{1}{N^{\frac{3}{2}}}\bigg),
            \end{equation}
\end{enumerate}
see e.g., \cite{ferbercombinatorics} also note that $O$ and $\Theta$ are asymptotically true in this case.
\section{System Identification via single trajectory}
\label{sec:sys-id-diag}
\subsection{Ordinary Least Squares estimator}
In this section we analyse the problem of OLS estimation for system transition matrix $A$ from single observed (as in \cite{sarkar2019near}, \cite{simchowitz2018learning}, \cite{tsiamis2021linear}) trajectory of $(x_0,x_1, \ldots,x_{N})$ satisyfing:
\begin{equation}
    \label{eq:LGS}
    x_{t+1}=Ax_{t}+w_{t}, \hspace{10pt} \text{ where } w_{t} \thicksim N(0,I) .
\end{equation}  
Before delving into solution of the estimation problem, we would like to give a brief overview into working of general OLS regression along with potential limitations:
A priori you are given $k-$ basis functions of an $n-$ dimensional vector space, where $n>k$ and one can form an $n \times k$ matrix $X:=[v_1,v_2,\ldots,v_{k}] \in \mathbb{R}^{n \times k}$. Now one observes $y$ :
\begin{equation}
    y=X\beta^{*}+\epsilon, \hspace{5pt} \epsilon \thicksim N(0,I_{n})
\end{equation}
and tries to estimate $\beta:=\beta(X,y) \in \mathbb{R}^{k \times 1}$ by projecting observation $y$ onto the span of $X$ as in \eqref{eq:projcoeff}, we get $\beta=(X^{*}X)^{-1}X^{*}y$ and the expected error in $\ell_{2}$ norm is:
\begin{align}
    \label{eq:lindpdata}
   \mathbb{E}\| \beta-\beta^{*} \|^{2}= Tr([X^{*}X]^{-1})=\sum_{j=1}^{k} d^{-2}_{j},
\end{align}
where last equality follows from negative second moment identity from Theorem \ref{thm:neg_2nd_moment_ineq} and one can immediately conclude: if any column of $X$ gets closer in terms of $\ell_{2}$ distance on $\mathbb{R}^{n}$ to the span of remaining $k-1$ columns, expected squared error in OLS estimation will increase. Vaguely speaking, linear dependence between basis of the data matrix deteriorates the performance of standard OLS regression. 
 
OLS solution for identification of LTI system from a single observed trajectiry is:
\begin{equation}
    \label{eq:OLSsol} \hat{A}= \arg \min_{B \in \mathbb{R}^{n \times n}}  \sum_{t=0}^{N-1} \|x_{t+1}-Bx_{t}\|^{2}.
\end{equation}
Let $X_{+}=[x_1, x_2,  \ldots, x_N]$ and $ X_{-}=[x_0, x_1, \ldots, x_{(N-1)}]$, noise covariates $E=[w_0, w_1, \ldots, w_{N-1}]$, $y_{j}$ be the rows of $X_{-}$ and $v_{j}$ be the hyperplane as defined in theorem \ref{thm:neg_2nd_moment_ineq}. Also notice that conditioned on $x_{0}=0$ state at time $i$ can be represented in terms of powers of $A$ and noise covariates as:  
\begin{equation}
\label{eq:dynamgauss}
x_{i}= \sum_{t=1}^{i}A^{i-t}w_{t-1}, \hspace{5pt} \textit{and graphical representation}    
\end{equation}
\[A^{*}:=
\begin{bmatrix}
    \vert & \vert & \vert & \vert \\
    \vert & \vert & \vert & \vert \\
    b_{1} &  b_{2} & \ldots & b_{n}   \\
    \vert & \vert & \vert & \vert \\
    \vert & \vert & \vert & \vert
\end{bmatrix}
\]

\[X^{*}_{-}=
\begin{bmatrix}
    \text{---} & x_{0}^{*} & \text{---} \\
    \text{---} & x_{1}^{*} & \text{---} \\
    \text{---} & \vdots & \text{---} \\
    \text{---} & x_{N-1}^{*} & \text{---} 
\end{bmatrix}
=
\begin{bmatrix}
    \vert & \vert &  \vert \\
    \vert & \vert &  \vert \\
    y_{1} &\ldots &    y_{n}   \\
    \vert & \vert &  \vert \\
    \vert & \vert &  \vert
\end{bmatrix}
\]

\[X^{*}_{+}=
\begin{bmatrix}
    \vert & \vert  & \vert \\
    \vert & \vert  & \vert \\
    z_{1} & \ldots & z_{n}   \\
    \vert & \vert & \vert \\
    \vert & \vert & \vert 
\end{bmatrix}
\]
In this dynamical version of OLS, at most $n$ basis are provided by columns of $X^{*}_{-}$. Then $i-$ th column of $A^{*}$ corresponds to co-effiecients estimated, by orthogonally projecting observation $z_{i}$ onto span of $X^{*}_{-}$
\begin{equation}
    \hat{b}_{i}=\big( X_{-} X^{*}_{-} \big)^{-1} X_{-}z_{i}
\end{equation}

Then the closed form expression for  Least squares solution and error are:
\begin{align}
 & \label{eq:OLS}    \hat{A}= X_{+}X_{-} ^{\dagger}, \hspace{5pt} \textit{where} \hspace{3pt} X_{-} ^{\dagger}:= X_{-} ^{*}(X_{-}X_{-}^{*})^{-1}  \\ & \label{eq:OLSerror} \big\|A-\hat{A}\big\|_{F} =\big\|EX_{-}^{\dagger}\big\|_{F}
\end{align}
\subsection{Geometric and spectral approaches to error analysis }
Spectral statistics of a random matrix of i.i.d Gaussian ensembles $E$ (each entry of the matrix is normally distributed with mean $0$ and variance $1$) has been very well studied in an attempt to leverage upon this information, authors in \cite{naeem2023high} managed to bound the error in terms of distances between row vectors and span of remaining rows. We reiterate there argument, by definition $\big\|A-\hat{A}\big\|_{F}^{2}= \sum_{k=1}^{n} \sigma_{k}^{2}\big(EX_{-} ^{\dagger}\big)$, using Courant-Fischer:
\begin{align}
    & \label{eq:errvianeg} \sigma_{n}\big(E_{N(X_{-})^{\perp}}\big) \sigma_{k}\big(X_{-} ^{\dagger}\big)   \leq  \sigma_{k}\big(EX_{-} ^{\dagger}\big) \leq \sigma_{1}\big(E_{N(X_{-})^{\perp}}\big)\sigma_{k}\big(X_{-} ^{\dagger}\big)
\end{align}
Recall from Lemma \ref{lm:dimred}, $X_{-} ^{\dagger}$ is a bijection from  $Im(X_{-})= \mathbb{R}^{n}$ to $N(X_{-})^{\perp}$ (subspace orthogonal to null space of $X_{-}$) and $dim[N(X_{-})^{\perp}]=n$, so given $N(X_{-})$ one can construct $N(X_{-})^{\perp}$ such that  $E_{N(X_{-})^{\perp}}:=EX_{-} ^{\dagger}=$ span$[w_{0}, w_{1}, \ldots, w_{n-1}]$, i.e., it should be viewed as an $n \times n$ square matrix of i.i.d Gaussian ensembles. Furthermore,
\begin{flalign}
    \big\|X_{-} ^{\dagger}\big\|_{F}^{2}\vspace{3pt} & \nonumber =Tr\big(X_{-} ^{\dagger}[X_{-} ^{\dagger}]^{*}\big)  =Tr\big((X_{-}X_{-}^{*})^{-1}\big)  = \sum_{j=1}^{n} \sigma_{j}^{-2}(X_{-}).
\end{flalign}
using negative second moment identity:
\begin{align}
    & \label{eq:ubderrlst}  \sigma_{n}(E^{\perp})\bigg(\sum_{j=1}^{n} d^{-2}_{j}\bigg)^{\frac{1}{2}}  \leq \big\|A-\hat{A}\big\|_{F} \leq  \sigma_{1}(E^{\perp})\bigg(\sum_{j=1}^{n} d^{-2}_{j}\bigg)^{\frac{1}{2}} ,    
\end{align}
where we have used $E^{\perp}$ as shorthand for $E_{N(X_{-})^{\perp}}$. 
Preceding bounds on error analysis seems favorable when underlying dynamics are generated from hermitian state transition matrix, as analysis of $\bigg(\sum_{j=1}^{n} d^{-2}_{j}\bigg)^{\frac{1}{2}}$ can potentially be decoupled into analysis of individual $d_{j}$ with some cost, because rows of the data matrix can be considered independent of each other.     

\textbf{Estimation error is independent of basis function choice}
Recall that if $A=A^{*}$, then there exists a unitary matrix $U \in \mathbb{C}^{n \times n}$ and a real diagonal matrix $\Omega \in \mathbb{R}^{n \times n}$ such that $A=U^{*} \Omega U$ and a coordinate transform for \eqref{eq:LGS}, i.e., $\hspace{3pt} z_{t+1}=\Omega z_t+Uw_t $, where $z_{t}:=Ux_{t}$ and $Uw_t$ is again an isotropic Gaussian.  Now let $W:=UE$, and notice that rows of $Z_{-}:=U X_{-}$ are independent one dimensional Linear Gaussian dynamics with growth/ decay parameter defined by corresponding diagonal entery in $\Omega$. As unitary by definition $U^* U=UU^*=I$, implies $U^{-1}=U^{*}$, we have 
\begin{align}
     & \nonumber \big\| WZ_{-}^{*} \big(Z_{-}Z_{-}^{*}\big)^{-1} \big\|_{F}^{2}=\big\| UE(UX_{-})^{*}(UX_{-}X_{-}^{*}U^{*})^{-1} \big\|_{F}^{2} \\ & \nonumber = Tr(EX_{-}^{*}(X_{-}X_{-}^{*})^{-2}X_{-}E^{*})=\|EX_{-}^{*}(X_{-}X_{-}^{*})^{-1}\|_{F}^{2}.  
\end{align}
So a first step towards rigorous analysis along bound in \eqref{eq:ubderrlst} would require quantifying distance between a random one dimensional ARMA trajectory of length $N$ and a fixed $n-1$ dimensional subspace of $\mathbb{R}^{N}$ with high probability and we study it in Section \ref{sec:concmdistsigma1}. 
Roughly speaking, analysis provided in \cite{simchowitz2018learning} and \cite{sarkar2019near} focuses on upper bounding the estimation error in operator norm by decomposing the error into a sample covariance term $\big(X_{-}X_{-}^{*}\big)^{-\frac{1}{2}}$ and a martingale difference term $EX_{-}\big(X_{-}X_{-}^{*}\big)^{-\frac{1}{2}}$. Again using the Courant-Fischer argument as in \eqref{eq:errvianeg}, we can now instead bound the Frobenius norm error in terms of the martingale difference term and extreme singular values of the data matrix:  
\begin{align}
    \label{eq:srkfill} \frac{\big\| EX_{-}^{*}\big(X_{-}X_{-}^{*}\big)^{-\frac{1}{2}}\big\|_{F}}{\sigma_{1}(X_{-})} \leq \big\|A- \hat{A}\big\|_{F} \leq \frac{\big\| EX_{-}^{*}\big(X_{-}X_{-}^{*}\big)^{-\frac{1}{2}}\big\|_{F}}{\sigma_{n}(X_{-})}.
\end{align}   
Work of \cite{sarkar2019finite} and \cite{simchowitz2018learning}, focused on bounding the right hand side of \eqref{eq:srkfill} in operator norm, which along with controlling martingale term required showing a high probability lower bound on $\sigma_{n}(X_{-})$. However, in order to ensure error convergence or divergence, error formulation in \eqref{eq:srkfill} would require understanding the behavior of largest singular value of the data matrix or said in another way largest eigenvalue of sample covariance matrix as $\sigma_{1}(X_{-})=\sqrt{\lambda_{\max}(X_{-}X_{-}^{*})}$. In fact, \cite{sarkar2019finite} shows inconsistency of explosive system with state transition matrix $A=1.1 I_{n}$ by showing largest singular value grows exponentially, see proposition 19.1 of \cite{sarkar2019near}. Recognizing the ubiquity of error formulation based on a martingale term as in \eqref{eq:srkfill}, for various system identification problems (stochastic system identification in \cite{tsiamis2022online}, Ho-Kalman algorithm in \cite{oymak2021revisiting} ) and no systematic procedure for the evloution and behavior of largest singular value (in terms of $n$ and $N$), we propose a method that lies at the  intersection of concentration of measure and behavior of finite powers of $A$: discussed in second part of Section \ref{sec:concmdistsigma1}. 


\subsection{From spectral theorem to statistical independence between the row blocks of the data matrix }
\label{subsec:Sp-gl}
Notice that using the description of linear transformation in Theorem \ref{thm:A-sp-K-l}, $i-$ the column of the data matrix can be decomposed into row blocks that are statistically independent of each other:  
\begin{align}
 \nonumber x_{i}= \sum_{t=1}^{i}A^{i-t}w_{t-1}&=\sum_{t=1}^{i} A^{i-t}\bigg( \sum_{m=1}^{K} P_{\lambda_{m}}\bigg) w_{t-1} \\ & \label{eq:sdmx} = \sum_{m=1}^{K}  \underbrace{\sum_{t=1}^{i}A_{\lambda_{m}}^{i-t} w_{t-1}^{m}}_{:=B_{\lambda_{m}}(i)}.
\end{align}

Now recall that $D_{\lambda_{m}}$ was used to denote discrepancy between algebraic and geometric multiplicity of eigenvalue $\lambda_{m}$ where for fixed $t$ notice that $\mathbb{E}\|w_{t-1}^{m}\|^{2}=\mathbb{E}\big \langle w_{t-1},P_{\lambda_{m}}w_{t-1} \big\rangle=Tr(P_{\lambda_m})=D_{\lambda_{m}}+1$. So $w_{t-1}^{m}$ is standard normal on $A-$ invariant subspace $M_{\lambda_{m}}$ of dimension $D_{\lambda_{m}}+1$. Since for all $t \in \mathbb{N}$ and $m,m^{'} \in [K]$ such that $m \neq m^{'}$, $w_{t}^{m}$ and $w_{t}^{m^{'}}$ are independent: $\mathbb{E}\big[w_{t}^{m} (w_{t}^{m^{'}})^{*}\big]=\mathbb{E}\big[P_{\lambda_{m}}w_{t}(P_{\lambda_{m^{'}}}w_{t})^{*}\big]=\mathbb{E}\big[P_{\lambda_{m}}w_{t}w_{t}^{*}P_{\lambda_{m^{'}}}\big]= P_{\lambda_{m}}P_{\lambda_{m^{'}}}= \delta_{m}(m^{'})I_{n}$, combined with the fact that $A_{m}^{*} A_{m'}=0$ for $m \neq m'$ (because $N(A_{\lambda_{m}}^{*})=\oplus_{j \neq m}^{K} M_{\lambda_{j}}$) implies that $B_{\lambda_{m}}(i)$ and $B_{\lambda_{m'}}(i)$ are independent of each other for all $i$. Therefore, data matrix essentially contains time realization of $K$-low dimensional dynamical systems, which are statistically independent of each other. As we showed in the previous section that estimation error is independent of the choice of underlying basis, w.l.o.g we can take canonical basis implying rows in data matrix comprises of independent blocks, where each block is a time realization of trajectory generated via canonical form of linear transformation with specified eigenvalue and size of the block equals 
 $D_{\lambda_{m}}+1$.

\section{Strong spatio-temporal correlations leading to covariates suffering from curse of dimensionality}
\label{sec:spatio-temp}
In order to unravel spatio-temporal correlations, it is important to represent elements of the data matrix in a compact form which one can do with inner products; $i-$ th column(time index) and $j-$ th row(space index) of the data matrix can be represented as
\begin{align}
    & \nonumber \big[X_{-}\big]_{j,i}=\big\langle x_{i},e_{j} \big \rangle= \sum_{t=1}^{i}\big \langle A^{i-t}w_{t-1},e_{j} \big \rangle.  
\end{align}
What makes the case of $n$ dimensional S-w-SSCs so peculiar is that covariate in $j$ th row is a function of row $[j,\ldots,n]$ of Gaussian ensemble $E$, precisely expressed:
\begin{theorem}
$A=J_{n}(\lambda)$ i.e., underlying dynamics are S-w-SSCs, then element corresponding to $j-$ th row and $i-$th column can be concisely expressed as:
\begin{align}
& \label{eq:causalrowdep}
[X_{-}]_{j,i}= \sum_{t=1}^{i} \sum_{m=0}^{(i-t) \land (n-j) } \binom{i-t}{m} \lambda^{i-t-m} \big\langle w_{t-1} ,e_{m+j} \big \rangle 
\end{align}
\end{theorem}
\begin{proof}
This is where inner product representation really helps along with noticing that nilpotent matrix is a shift operator, as:
\begin{align}
& \nonumber  \sum_{t=1}^{i} \langle A^{i-t}w_{t-1},e_{j}\rangle  = \sum_{t=1}^{i} \big\langle (\lambda I+N_{n})^{i-t} w_{t-1},e_{j}\big \rangle \\  \nonumber&= \bigg\langle \sum_{t=1}^{i} \sum_{m=0}^{(i-t)\land n}\binom{i-t}{m}N_{n}^{m} \lambda^{i-t-m} w_{t-1},e_{j}\bigg\rangle 
\end{align}
Now a simple observation reveals that $N_{n}^{m^{*}}e_{j}=e_{j+m}$ for $m \leq n-j$, and $0$  otherwise. Hence,
\begin{align}
     [X_{-}]_{j,i}=\sum_{t=1}^{i} \sum_{m=0}^{(i-t)\land(n-j)}\binom{i-t}{m} \lambda^{i-t-m}\bigg\langle  w_{t-1},e_{j+m}\bigg\rangle.
\end{align}
\end{proof}
Now one can trivially see that each covariate in $j-$ th row is a weighted sum of i.i.d standard normals that make up rows $[j, \ldots, n]$ of Gaussian Ensemble. Furthermore, all these weights are positive, so as one would suspect that first row is most susceptible to having the largest typical size. Which turns out to be exponential in dimension of the state space 
\begin{lemma}
    \label{lm:cov-ul-bd}
    $N > n$ and $\lambda \in (\frac{1}{2},1)$ then typical order of first row of the data matrix corresponding to dynamics generated from $n$ dimensional S-w-SSCs is exponential in $n$, precisely said
\begin{flalign}
 & \nonumber Var(y_1) \preccurlyeq \sum_{i \geq n}^{N} \sum_{m=0}^{n-1} \lambda^{2(i-n+1)} \sum_{l=1}^{i-n} \zeta_{m,l}(i,\lambda)\\ & \nonumber +(N-n+1)O_{\lambda}(4^n) \hspace{3pt} \textit{and} \hspace{3pt} Var(y_1) \succcurlyeq (N-n+1)\Omega(4^{n} \lambda^{2n}) \\ & \nonumber  + \sum_{i>n}^{N} \sum_{m=0}^{n-1} \lambda^{2(i-m)} \sum_{l=1}^{i-n} \zeta_{m,l}(i,\lambda) \exp{\Bigg(-O\bigg(2\frac{m^{3}}{(i-l)^{2}}\bigg)\Bigg)} 
\end{flalign}
where $\zeta_{m,l}(N,\lambda):=\frac{1}{m!m!} \frac{(N-l)^{2m}}{\lambda^{2l}} \exp{\big(-\frac{m(m-1)}{N-l}\big)}$
\end{lemma}
See Appendix \ref{subsec:pfl4} for complete proof. As we have convinced ourselves that data matrix can be divided into row blocks that are statistically independent except for the S-w-SSCs case, along with the compact representation for individual elements inside the data matrix, we now aim at understanding elementwise estimation error.  
\section{An 'almost' exact solution for OLS regression on LTI systems }
\label{sec:exactOLS}
Although, bounds on estimation error discussed in preceding section provide a good intuition, but existing analysis based on martingales e.t.c does not reveal explicit dependencies on number of iterations and state space dimensions. Surprisingly, it had been left unnoticed that pseudo-inverse is constrained with respect to the data matrix which we discussed is itself a function of Gaussian ensemble $E$ and element-wise estimation error is an inner product between Gaussian ensemble and pseudo-inverse. 

\begin{theorem}
\label{thm:invcovneg2ndmoment}[Construction of Inverse Sample Covariance Matrix]
Let $[e_{j}]_{j=1}^{n}$ be canonical basis of $\mathbb{C}^{n}$ and $X_{-}$ be the data matrix of stable linear dynamical system with isotropic Gaussian noise as in \eqref{eq:LGS}, Given, the rows of data matrix $[y_{j}]_{j=1}^{n}$, inverse sample covariance matrix $\big(X_{-}X_{-}^{*}\big)^{-1}=[v_{j,k}]_{j,k=1}^{n}$ satisfies following constraints:
\begin{enumerate}
    \item For fixed $j$ in $[1, \ldots, n]$, $\sum_{k \neq j} v_{j,k} \big\langle y_{k},y_{j} \big\rangle=1- v_{j,j}\|y_{j}\|^{2}$, precisely said:
    \item and for all $l \neq j$, $\sum_{k=1}^{n} v_{j,k} \big\langle y_{k},y_{l} \big\rangle=0$, precisely said:
    \item diagonal entries of inverse sample covariance matrix are inverse  distance squared  between row and conjugate hyperplane, to be precise: 
    \begin{align}
    \label{eq:invsmpdiag} \bigg \langle (\sum_{t=0}^{N-1}x_{t}x_{t}^{*})^{-1} e_{j}  ,e_{j}\bigg\rangle= \frac{1}{d^{2}(y_j,v_j)}=v_{j,j},  
    \end{align}
\end{enumerate} 
\end{theorem}
Notice that total number of unknowns $v_{j,k}$ are $\frac{n(n+1)}{2}$ which equals distinct interaction potentials of observations$\langle y_{j}, y_{k}\rangle$. A remarkable advantage of enlisting these constraints is when we do not have independence between the rows as in Hermitian case, we still have:
\begin{align}
     & \nonumber d_{j}^{-2}=\frac{1-\sum_{k \neq j} v_{j,k} \langle y_{k} , y_{j} \rangle}{\|y_{j}\|^{2}},\textit{where} \hspace{3pt}  \sum_{k \neq j} v_{j,k} \langle   y_{k} , y_{j} \rangle \leq 0 \hspace{3pt} \textit{and} \\ & \nonumber \sum_{j=1}^{n} d_{j}^{-2}= \frac{\sum_{j=1}^{n} \prod_{l \neq j} \|y_{l}\|^{2}\big(1-\sum_{k \neq j} v_{j,k} \langle y_{k} , y_{j} \rangle \big)    }{\prod_{j=1}^{n} \|y_{j}\|^{2}}  
\end{align}

As we will see shortly afterwards, elementwise estimation error is an inner product between the rows of the Gaussian ensemble and and columns of the pseudo-inverse, it will be helpful to translate preceding constraints into constraints on columns of pseudo-inverse   

\begin{corollary}
    [Exact Error in Frobenius norm] Let $\hat{e}_{i}$ be the canonical basis of $\mathbb{C}^{N}$
\begin{align}
    \|A-\hat{A}\|_{F}=\sqrt{\sum_{j,k=1}^{n} \bigg|\sum_{i=1}^{N} \langle w_{i},e_{j}\rangle \langle c_{k},\hat{e}_{i}\rangle \bigg|^{2}},
\end{align}
    where for each $k \in [1,2,\ldots,n]$, $c_{k} \in \mathbb{C}^{N}$ and satisfies:
\begin{equation}
\label{eq:normaleq}
    \sum_{i=1}^{N} \bigg[ \sum_{t=1}^{i} \langle A^{i-t}w_{t-1},e_{j}\rangle  \bigg]\langle c_{k}, \hat{e}_{i}\rangle= \delta_{j}(k),
\end{equation}
for every $j \in [1,2,\ldots,n]$
\end{corollary}
\begin{proof}
Also notice that $c_{k}:=X_{-}^{*}(X_{-}X_{-}^{*})^{-1}e_{k}$ and
\begin{equation}
    \big\|c_{k}\|^{2} =\big\langle X_{-}^{*}(X_{-}X_{-}^{*})^{-1}e_{k},X_{-}^{*}(X_{-}X_{-}^{*})^{-1}e_{k}\rangle=\frac{1}{d^{2}(y_k,n_k)}.
\end{equation}
\end{proof}
Now we simply notice that elementwise error is an inner product between rows of Gaussian ensemble and a constrained random variable.
\begin{remark}
\label{rmk:inp}
     Element-wise estimation error of $A$, $[EX_{-}^{*}(X_{-}X_{-}^{*})^{-1}]_{j,k}$ for each $j,k \in [n]$  is an inner product, hence a scalar weighted random walk with weight functions coming from pseudo-inverse.    
\begin{align}
    \sum_{i=1}^{N}[E]_{j,i} \langle c_{k},\hat{e}_{i}\rangle  =\underbrace{\sum_{i=1}^{N} \langle w_{i},e_{j}\rangle \langle c_{k},\hat{e}_{i}\rangle}_{\textit{Littlewood-Offord problem}}, 
\end{align}
Furthermore, total error in Frobenius norm is a random polynomial:
\begin{align}     
      \label{eq:excterrorelemenfrob} \sum_{j,k=1}^{n} \sum_{i=1}^{N} \bigg[ |\langle w_{i},e_{j}\rangle \langle c_{k},\hat{e}_{i}\rangle|^{2}\nonumber+  \langle w_{i},e_{j}\rangle \langle c_{k},\hat{e}_{i}\rangle \sum_{l \neq i} \overline{\langle w_{l},e_{j}\rangle \langle c_{k},\hat{e}_{l}\rangle} \bigg]
\end{align}
\end{remark}

\emph{In a situation where a solution of $c_{k}$ is independent of the $j-$th row of Gaussian ensemlbe $E$, after conditioning $\sum_{i=1}^{N} \langle w_{i},e_{j}\rangle \langle c_{k},\hat{e}_{i}\rangle$ is similar to a scalar random walk}, which is well studied under the name of \emph{Littlewood Offord problem} and its' typical behavior is sensitive to the structure of constants $c_{k}$, as discussed in the Section \ref{sec:tensorization}. On the other hand if $c_{k}$ is itself some function of $[\langle w_{i},e_{j}\rangle]_{j=1}^{N}$, bounding the error is atleast a \emph{Quadratic variant of the Little-wood Offord problem}(see e.g., \cite{costello2006random}). So understanding dependence of constants on covariates of $E$ will be imperative in getting deep insights into working of OLS.

\textbf{Example 2.} Now we consider two extreme cases of stable dynamical systems and show how contraints on pseudoinverse have distinct dependencies on the rows of Gaussian ensemble in each case:
$A=D$, diagonal block with only one eigenvalue $\lambda$, fix $k \in [n]$ then constraints on pseudo-inverse from \eqref{eq:normaleq} imply:
\begin{align}
&\nonumber \sum_{i=1}^{N} [X_{-}]_{j,i} \langle c_{k}, \hat{e}_{i}\rangle=
\sum_{i=1}^{N} \sum_{t=1}^{i} \lambda^{i-t} \big \langle w_{t-1},e_{j}\rangle\big \langle c_{k}, \hat{e}_{i}\rangle= \delta_{j}(k)
\end{align}
Notice that for fixed $k$, preceding constraint is only a function of the $j-$th row of Gaussian ensemble
$A=J_{n}(\lambda), \delta_{j}(k)=$ 
\begin{equation}
\label{eq:normaljordan_jk} \sum_{i=1}^{N}   \sum_{t=1}^{i} \sum_{m=0}^{(i-t)\land(n-j)}\binom{i-t}{m} \lambda^{i-t-m}\langle  w_{t-1},e_{j+m}\rangle \langle c_{k}, \hat{e}_{i}\rangle
\end{equation}
        
\begin{remark}
\label{rmk:distinctqual}
    Notice that in stark contrast to constraint from the diagonal setting , constraint equation corresponding to $\delta_{j}(k)$ when dynamics are generated from an $n \times n$ Jordan block \eqref{eq:normaljordan_jk} is an inner product between $k-$ th column on pseudoinverse $c_{k}$ and some function  dependent on rows $[j,\ldots,n]$ of the Gaussian ensemble $E$. Implying that the $j-$th row of the data matrix $X_{-}$ is spatially correlated with rows $[j,j+1,\ldots,n]$ of data matrix, or said in another way: $j-$ th row of the data matrix is a function of rows $[j,j+1, \ldots,n]$ of the Gaussian ensemble E. This is an unexpected yet remarkable advantage of using inner products. 
\end{remark}

\section{ Distance between a random row and hyperplane and  Largest singular value of the data matrix }
\label{sec:concmdistsigma1}
These distance estimates have been an integral part of recent breakthrough progress in Random Matrix Theory. However, the problem under dynamical systems setting is much more complicated as: observations in every row are correlated and different rows are themselves correlated. We will first focus on dealing the former problem. A good starting point is quantifying distance between a random row of length $N$ and an $n-1$ dimensional subspace of $\mathbb{R}^{N}.$

\begin{theorem}
Let $\mathcal{V}$ be a fixed $n-1$ dimensional subspace of $R^{N}$, and random vector $x \thicksim N(0,I_{N})$
where we assme that $N>n$, with probability 1 distance function concentrates around $[(N-n+1) -\sqrt{2(N-n+1)}, (N-n+1) + \sqrt{2(N-n+1)}]$.
\end{theorem}
\begin{proof}
Let $P$ be the projection matrix associated with the subspace $\mathcal{V}$. By the definintion of projection matrix we have that $P \in \mathbb{R}^{N \times N}$ and   $Tr(P)=\sum_{i=1}^{N} p_{ii}=n-1$ and $P=P^{*}=P^2$. Let $A=P-$ Diag($p_{11}, \ldots, p_{NN}$), so diagonal elements of $A$: $(a_{ii})_{i=1}^{N}$ are zeros and non-diagonal are same as projection matrix. Pythagoreous theorem reveals:
\begin{align}
   \nonumber d^{2}(x,\mathcal{V})  =\|x\|^{2}- \langle x ,P x\rangle =\sum_{i=1}^{N}(1-p_{ii})x_{i}^{2} -\sum_{j,k=1}^{N} a_{jk}x_{k}x_{j}
\end{align}
Taking expectation, while keeping in mind $\mathbb{E}[x_{k}x_{j}]=0$ for $k \neq j$ and trace equality, reveals:
\begin{equation}
    \mathbb{E}[d^{2}(x, \mathcal{V})]=N-n+1
\end{equation}

However, in order to use Chebyshev inequality that will lead to concentration result for $d^{2}(x,V)$, we need to compute:
\begin{align}
&Var[d^{2}(x,V)]=\mathbb{E} d^{4}(x,V)-(N-n+1)^{2} \nonumber=\\ & \nonumber \mathbb{E}\bigg[ \big(\sum_{i=1}^{N}(1-p_{ii})x_{i}^{2} -\sum_{j,k=1}^{N} a_{jk}x_{k}x_{j}\big)^{2}\bigg]-(N-n+1)^{2} \\ & \label{eq:Stein}= \sum_{i=1}^{N} \bigg[3(1-p_{ii})^{2}+ 2(1-p_{ii}) \sum_{j>i} (1-p_{jj})\bigg]-(N-n+1)^{2} \\ & +\nonumber \mathbb{E}[Y^{2}].  
\end{align}
Where \eqref{eq:Stein} follows from \emph{Stein's Lemma}; if $w \thicksim \mathcal{N}(0, \sigma^{2})$, then for any odd power $k$,  $\mathbb{E} w^{k}=0 $ and for all $k \in \mathbb{N}$ we have higher moments  $\mathbb{E}w^{2k}=\sigma^{2k}\prod_{l=1}^{k} (2l-1)$. Let $Y:=\sum_{j,k=1}^{N} a_{jk}x_{k}x_{j}$ and notice that $\mathbb{E}(Y^2)=2Tr(A^2)$, where:
\begin{align}
     Tr(A^2) & \nonumber=Tr(P^2)- \sum_{i=1}^{N} p_{ii}^2=\sum_{i=1}^{N}\sum_{j=1}^{N} p_{ij}^{2}-\sum_{i=1}^{N}p_{ii}^2 \\ & \nonumber= Tr(P)-\sum_{i=1}^{N}p_{ii}^2 
\end{align}

The remaining terms in \eqref{eq:Stein},
\begin{align}
    & \nonumber \sum_{i=1}^{N} [3(1-p_{ii})^{2}+ 2(1-p_{ii}) \sum_{j>i} (1-p_{jj})] \\ & \nonumber= \sum_{i=1}^{N} 2(1-p_{ii})^2+ [Tr(I-P)]^2= \sum_{i=1}^{N} 2(1-p_{ii})^2\\ & \nonumber +(N-n+1)^2.
\end{align}
But,
\begin{align}
& \nonumber \sum_{i=1}^{N} 2(1-p_{ii})^{2}+2Tr(A^2)=2[N-n+1] \\ & \label{eq:vard2ONn1} \textit{Therefore,} \hspace{5pt} Var[d^{2}(x, \mathcal{V})]= 2(N-n+1) 
\end{align}
Consequently, Tchebyshev implies with probability 1,
$d^{2}(x, \mathcal{V}) \in [(N-n+1) -\sqrt{2(N-n+1)}, (N-n+1) + \sqrt{2(N-n+1)} ]$
\end{proof}
\subsection{Distance between a stable trajectory and a fixed subspace}   
Now instead of each element of the row being independent, it now corresponds to 
one dimensional ARMA trajectory of length $N$ as in \eqref{eq:1dimarma} and its' distance from a fixed subspace $V$ is: 
\begin{align}
   \nonumber d^{2}(x,\mathcal{V}) =\sum_{i=1}^{N}(1-p_{ii})x_{i}^{2} -2\sum_{j=1}^{N}\sum_{k>j}^{N} p_{jk}x_{k}x_{j}.
\end{align}
Exploiting the Markovian structure of the dynamics for $k>j$:
\begin{align}
      \nonumber x_{k} &=\lambda^{[k-j]}x_{j}+\lambda^{[k-j]-1}w_{j}+ \lambda^{[k-j]-2}w_{j+1}+\ldots+w_{k-1} \\  \nonumber x_{k}x_{j}&=\lambda^{[k-j]}x_{j}^2\\ & + \big(\lambda^{[k-j]-1}w_{j}+\lambda^{[k-j]-2}w_{j+1}+\ldots+w_{k-1}\big)x_{j}. 
\end{align}
Trace and elements of the covariance matrix $\Sigma_{N,\lambda}[k,j]:=\mathbb{E}[x_{k}x_{j}]$ are
\begin{align}
     & \label{eq:sumvar}
      Tr(\Sigma_{N, \lambda})= \sum_{j=1}^{N} \Sigma_{N, \lambda}[j,j]=  \mathbb{E} \big[\sum_{i=1}^{N} x_{i}^{2} \big]= \sum_{i=1}^{N} i \lambda^{2(N-i)}\\ & \label{eq:covarmark} \Sigma_{N,\lambda}[k,j]=\lambda^{[k-j]} \mathbb{E}[x_{j}^2]=\lambda^{[k-j]}(1+\lambda^2+\ldots+\lambda^{2[k-j]-1}), 
\end{align}

\begin{figure*}

\label{fig:covmat}
\begin{align*}
\Sigma_{N,\lambda}:=
\begin{pmatrix}
1 & \lambda & \lambda^2 & \lambda^3 & \ldots  \\
\lambda & 1+\lambda^{2}  & \lambda^3+\lambda & \ldots & \ldots  \\
\lambda^2 & \lambda^3+\lambda & 1+\lambda^2+\lambda^4 & \lambda^{5}+\lambda^{3}+\lambda &    \lambda^{N+1}+\lambda^{N-1}+ \lambda^{N-3}     \\
 \ldots & \ldots & \ldots & \ldots\\
\lambda^{N-1} &\lambda^{N}+\lambda^{N-1} & \ldots & \ldots &\ldots& \lambda^{2(N-1)}+\lambda^{2(N-2)}+ \ldots +1 \\    
\end{pmatrix}
\end{align*}
\end{figure*}

In the presence of temporal correlation between elements of the row vectors we need a tight control on the spectrum of the covariance matrix, shown in  \ref{fig:covmat}.

\begin{theorem}
    There exists positive constants $c_{\lambda,1}, c_{\lambda,2} $ and $c_{\lambda,3}$ such that for all $N \in \mathbb{N}$
    \begin{equation}
    \label{eq:tracebd}
      c_{\lambda,1}N - c_{\lambda,2}  \leq Tr(\Sigma_{N, \lambda}) \leq c_{\lambda,1}N- c_{\lambda,2} +c_{\lambda,1}\bigg(1-\frac{c_{\lambda,3}}{N}\bigg).
    \end{equation}
Where $c_{\lambda,1}=\frac{|\lambda|^{-2}}{\ln{|\lambda|^{-2}}}$
     ,$c_{\lambda,2}=\frac{|\lambda|^{-2}}{[\ln{|\lambda|^{-2}}]^{2}}$ and $c_{\lambda,3}= \frac{1}{\ln{|\lambda|^{-2}} }$.
Consequently,
\begin{equation}
    \label{eq:firstmomentlipnorm} \nonumber  \Omega \big(   1   \big)   \leq \big\|\Sigma_{N, \lambda}^{\frac{1}{2}}\big \|_{2} \leq O(N^{\frac{1}{2}}).
\end{equation}
     Moreover, there exists a constant $d_{\lambda,1}$ and positive constant $d_{\lambda,2}$ such that: 
     \begin{align}
         &  \label{eq:frob-squareubd} \big\|\Sigma_{N,\lambda}\big\|_{F}^{2}=Tr(\Sigma_{N,\lambda}^2)  =d_{\lambda,2}N+ o_{\lambda;N}(1) +d_{\lambda,1},
    \end{align}   
where given $|\lambda|<1$, we use notation $o_{\lambda;N}(1)$ to represent a term dependent on $\lambda$ but going to zero as $N \rightarrow \infty$. Furthermore, \eqref{eq:frob-squareubd} implies even better control on $\|\Sigma_{N,\lambda}^{\frac{1}{2}}\|_{2}$:
\begin{equation}
\label{eq:Talnorm}
   \Omega(1) \leq \|\Sigma_{N,\lambda}^{\frac{1}{2}}\|_{2} \leq O(N^{\frac{1}{4}}).
\end{equation}
\end{theorem}

\begin{proof}
    Recall from equation \eqref{eq:sumvar}, $Tr(\Sigma_{N,\lambda})=\mathbb{E} \big[\sum_{i=1}^{N} x_{i}^{2} \big]= \sum_{i=1}^{N} i \lambda^{2(N-i)}$. Now we can control the value of trace by lower bounding it by area under an appropriate curve, similarly upper bounding it by area of some curve. Let $a:= \ln{|\lambda|^{-1}}$, then:
\begin{equation}
      \int_{0}^{N} xe^{2ax}dx  \leq \sum_{i=1}^{N} i \lambda^{-2i} \leq \int_{0}^{N} (x+1)e^{2a(x+1)}dx
\end{equation}    
Frobenius norm estimates follows by noticing:
\begin{align}
& \nonumber
\big\|\Sigma_{N, \lambda}\big\|_{F}^{2}=\sum_{j=1}^{N}  |\mathbb{E}[x_{j}^2]|^{2}+2\sum_{k>j}^{N} |\mathbb{E}[x_jx_k]|^2 \\ & \nonumber=2\sum_{j=1}^{N}\big[1+\sum_{k>j}^{N}|\lambda|^{2[k-j]}\big]|\mathbb{E}[x_{j}^{2}]|^2- \sum_{j=1}^{N}  |\mathbb{E}[x_{j}^2]|^{2} \\ & \nonumber= \frac{4N}{[1-|\lambda|^2]^3}|\lambda|^{2(N+1)} \nonumber -2|\lambda|^{2(N+1)} \sum_{j=1}^{N}\big[ |\lambda|^{2j} +|\lambda|^{-2j} \big]\\ & \nonumber + \frac{1}{[1-|\lambda|^2]^2} \big[ \frac{2}{[1-|\lambda|^2]} -1 \big]\sum_{j=1}^{N} \big[1-|\lambda|^{2j}\big]^2 \\ & \nonumber=\frac{4N}{[1-|\lambda|^2]^3}|\lambda|^{2(N+1)} -2 \frac{\big[1-|\lambda|^{2N}\big]|\lambda|^2}{1-|\lambda|^2}\big[1+|\lambda|^{2(N+1)}\big] \\ & \nonumber+\frac{1}{[1-|\lambda|^2]^2} \big[ \frac{2}{[1-|\lambda|^2]} -1 \big] \big[ N+ \frac{|\lambda|^{4}\big(1-|\lambda|^{4N} \big)}{1-|\lambda|^{4}}\\ & \nonumber \hspace{135pt} +\frac{|\lambda|^{2}\big(1-|\lambda|^{2N} \big)}{1-|\lambda|^{2}} \big]
\end{align}
\end{proof}
\begin{remark}
    [The moment method:] notice that how computing trace of higher powers of covariance matrix leads to better estimates of operator norm of covariance matrix. Let $\lambda_{1} \geq \lambda_{2} \geq \ldots \lambda_{N} > 0$ be the eigenvalues of $\Sigma_{N,\lambda}$, as trace of a matrix equals sum of its eigenvalues and operator norm of positive definite matrix equals largest positive eigenvalue , we have for any $k \in \mathbb{N}$:
    \begin{align}
       & \nonumber \big\|\Sigma_{N, \lambda}\big\|_{2}^{k}=\lambda_{1}^{k} \leq \sum_{i=1}^{N} \lambda_{i}^{k} =Tr(\Sigma_{N, \lambda}^{k}) \leq N \lambda_{1}^{k}=N \big\|\Sigma_{N, \lambda}\big\|_{2}^{k} \\ & \label{eq:trl2} \textit{implying}, \hspace{15pt} \frac{Tr(\Sigma_{N, \lambda}^{k})^{\frac{1}{k}}}{N^{\frac{1}{k}}} \leq \big\|\Sigma_{N, \lambda}\big\|_{2} \leq Tr(\Sigma_{N, \lambda}^{k})^{\frac{1}{k}}.
    \end{align}
That is by computing trace of higher powers of $\Sigma_{N, \lambda}$, one can get a tight control on estimate of $\|\Sigma_{N, \lambda}\big\|_{2}$ which will reveal order of the deviation of underlying \emph{quadratic form} of interest in terms of state space dimensions and number of iterations. It is important to point out that we believe that $\Omega(1) \leq \|\Sigma_{N,\lambda}^{\frac{1}{2}}\|_{2} \leq O(N^{\frac{1}{4}})$ in \eqref{eq:Talnorm} can be improved to $\|\Sigma_{N,\lambda}^{\frac{1}{2}}\|_{2}=\Theta(1)$ as suggested by dimension independent tensorization of Talagrands' inequality in Theorem \ref{thm:dim_ind_tal}, by computing higher powers of trace.     
\end{remark}
\begin{theorem}
    Let $x=(x_1,x_2, \ldots, x_N)$ be the trajectory of length $N$ from one dimensional ARMA model \eqref{eq:1dimarma}, with stable eigenvalue $\lambda$ and $V$ be an $n-1$ dimensional subspace of $\mathbb{R}^{N}$, then:
\begin{align}
    & \label{eq:edtwo} \mathbb{E}d^{2}(x,V)=Tr(\Sigma_{N,\lambda} P_{V^{\perp}}) \leq \|\Sigma_{N,\lambda}\|_{F}\|P_{V^{\perp}}\|_{F} \\ & \nonumber=O\bigg( \sqrt{\big[d_{\lambda,2}N+o_{\lambda;N}(1)+d_{\lambda,1}\big] (N-n+1)} \bigg).
\end{align}    
Furthermore, there exists positive constants $c_{\lambda}$ and $c$ such that for all $\delta>0$:
\begin{align}
    P\bigg(|d^{2}(x,V)-Tr(\Sigma_{N,\lambda} P_{V^{\perp}})| \geq & \delta c_{\lambda}N^{\frac{1}{2}} (N-n+1)^{\frac{1}{4}} \bigg) \\& \leq \exp{(-c \delta^2)},
\end{align}    
implying fluctuation of $d^2$ around its mean, are of order $O_{\lambda}(N^{\frac{1}{2}} (N-n+1)^{\frac{1}{4}})$
\end{theorem} 
\begin{proof}
We know from application of \emph{Cauchy-Schwarz}, $Tr(\Sigma_{N,\lambda} P_{V^{\perp}}) \leq \|\Sigma_{N,\lambda}\|_{F}\|P_{V^{\perp}}\|_{F}$. Orthogonal projections satisfy, $\|P_{V^{\perp}}\|_{F}:=\sqrt{Tr(P_{V^{\perp}} P_{V^{\perp}})}=\sqrt{Tr(P_{V^{\perp}})}=\sqrt{N-n+1}$ and first result in \eqref{eq:edtwo} follows from Frobenius norm control of $\Sigma_{N,\lambda}$ given in \eqref{eq:frob-squareubd}. \newline
From Talagrands' inequality, we know that:
    \begin{equation}
        P\bigg( \big|d(x,V) -\mathbb{E}d(x,V)\big| \geq \delta \big\|\Sigma_{N, \lambda}^{\frac{1}{2}}  \big\|  \bigg) \leq \exp{\big(-\delta^2\big)}.
    \end{equation}

Therefore, fluctuations of $d^{2}$ around its mean are at most of order $O\bigg(\big\|\Sigma_{N, \lambda}^{\frac{1}{2}}\big\|\mathbb{E}d(x,V)\bigg)$, result now follows by using Jensens' inequality and $\ell_{2}$ norm estimate of $\Sigma_{N, \lambda}^{\frac{1}{2}}$ given in \eqref{eq:Talnorm}.   
\end{proof}
\begin{remark}
Notice that $d^{2}(x,V)= \bigg\langle z_{N}, \Sigma_{N, \lambda}^{\frac{1}{2}} P_{V^{\perp}} \Sigma_{N, \lambda}^{\frac{1}{2}} z_{N} \bigg\rangle$ for $z_{N} \thicksim N(0,I_{N})$, so it is a quadratic form and we just saw its' deviation around mean is at most of order $N^{\frac{1}{2}} (N-n+1)^{\frac{1}{4}}$. Most of the literature on sample complexity of system identification type problems, study quadratic forms using Hanson-Wright inequality(\cite{tsiamis2022statistical}) which gives  deviation in terms of $\|\Sigma_{N, \lambda}^{\frac{1}{2}} P_{V^{\perp}} \Sigma_{N, \lambda}^{\frac{1}{2}}\|_{F}, \|\Sigma_{N, \lambda}^{\frac{1}{2}} P_{V^{\perp}} \Sigma_{N, \lambda}^{\frac{1}{2}}\|_{2}$; our result suggest using moment methods and approximation techniques (we used in preceding theorems
) to unravel potential hidden dependencies on dimensionality and iterations while using Hanson-Wright. 
\end{remark}


\subsection{Largest eigenvalue of sample covariance matrix}
A natural assumption for control on the spectrum of the data matrix via spectrum of rectangular Gaussian ensemblie is:
\begin{assumption} [Trace control of the sample covariance matrix]  
\label{asp:dataGaussFrob}
 Given dimensions of the underlying state space $n$, and length $N$ of the simulated trajectory in \eqref{eq:LGS}, there exists a positive constant $L_{n,N}$ such that:
 \begin{equation}
 \label{eq:a2}
    \sum_{i=0}^{N-1} \sum_{j=1}^{n} |\langle x_{i},e_{j}\rangle|^{2} \leq L_{n,N}^{2} \sum_{i=0}^{N-1} \sum_{j=1}^{n} |\langle w_{i},e_{j}\rangle|^{2}.
\end{equation}
\end{assumption}
Notice that, preceding assumption implies:
\begin{equation}
    \sum_{i=1}^{n} \sigma_{i}^{2}(X_{-}) \leq L_{n,N} \sum_{i=1}^{n} \sigma_{i}^{2}(E),
\end{equation}
 Verification of assumption \ref{asp:dataGaussFrob}, relies on the behavior of powers of $[A^{t}]_{t \in \mathbb{N}}$, as \eqref{eq:a2} is equivalent to showing:
\begin{equation}
\label{eq:lipmapdatatorectgauss}
    \sum_{i=0}^{N-1} \sum_{j=1}^{n} \big| \sum_{t=1}^{i}\langle A^{i-t}w_{t-1},e_{j}\rangle \big|^{2} \leq L_{n,N}^{2} \sum_{i=0}^{N-1} \sum_{j=1}^{n} | \langle w_{i},e_{j}\rangle|^{2}.
\end{equation}

\begin{theorem}
\label{thm:sigma1data}
    If the dynamical system with isotropic Gaussian excitations(LGS) satisfies assumption made in \ref{asp:dataGaussFrob}, then:
    \begin{enumerate}
        \item $\sigma_{1}(X_{-})$ is an $L_{n,N}$ Lipschitz map from - $n \times N$ Gaussian ensemble taking values in $\mathbb{R}^{n^{\otimes N}}$ equipped with $\ell_{2}$ metric as in Theorem \ref{thm:dim_ind_tal}, to $\mathbb{R}$.
        \item Furthermore, we have the following deviation bounds of the largest singular values of the data matrix
            \begin{align}
            \label{eq:S_1dev}
            P \bigg( \big|\sigma_{1}(X_{-})-\mathbb{E}\sigma_{1}(X_{-})\big| \geq \sqrt{2}\delta L_{n,N} \bigg) \leq 2e^{- \delta^2},
            \end{align}
        for every $\delta \in (0,1)$.
    \end{enumerate}

\end{theorem}
\begin{proof}
Conditioned to $x_{0}=0$ we know from \eqref{eq:dynamgauss} that data matrix is function of $n \times N $ Gaussian ensemble, $\big[\langle w_{i},e_{j}\rangle \big]_{i=0,j=1}^{N-1,n}$. So consider the following map: 
\begin{align}
& \nonumber \hspace{60pt} F:\big(\mathbb{R}^{n^{\otimes N}},\ell_{2}\big) \longrightarrow \mathbb{R}  \\ & \nonumber
[w_{0},w_{1},\ldots,w_{N-1}] \mapsto [x_{0},x_{1},\ldots,x_{N-1}] \mapsto \|X_{-}\| \\ & \label{eq:GaussSigma}
\hspace{10pt} F\big([\langle w_{i},e_{j}\rangle]_{i=0,j=1}^{N-1,n}\big):=\|X_{-}\|= \sigma_{1}(X_{-}),
\end{align}
combined with condition \eqref{eq:a2} and the fact that $\|E\|_{F}=\|E\|_{\ell_2}$, leads to: 
\begin{equation}
    \sigma_{1}(X_{-})=\|X_{-}\| \leq \|X_{-}\|_{F} \leq L_{n,N}\|E\|_{\ell_2},
\end{equation}
i.e.,  largest singular value of the data matrix is $L_{n,N}$ Lipschitz function from $\mathbb{R}^{n^{\otimes N}}$ to $\mathbb{R}$. As i.i.d normals satisfy dimension independent talagrands inequality as in Theorem \ref{thm:dim_ind_tal}, when combined with Lipschitz concentration in Remark \ref{rm: lipiid}, deviation bounds for largest singular value \eqref{eq:S_1dev} follows.  \end{proof}
\begin{corollary}
     Gordons' lemma (see e.g., \cite{vershynin2010introduction}) for Gaussian ensemble can be used to conclude even an upper bound on expected value of the largest singular value, :
\begin{equation}
    \mathbb{E}\sigma_{1}(X_{-}) \leq  L_{n,N}\sqrt{n} \mathbb{E}\sigma_{1}(E) \leq  L_{n,N}\sqrt{n} \big(\sqrt{N}+\sqrt{n}\big).
\end{equation}
Therefore, with probability one:
\begin{equation}
    \sigma_{1}(X_{-}) \leq L_{n,N} \big(\sqrt{2}+\sqrt{nN}+n\big). 
\end{equation}
\end{corollary}
Although, preceding deviation bounds for the largest singular value of the data matrix is non-trivial, but when is assumption \ref{asp:dataGaussFrob} valid? In fact to begin with it implies Lipschitz behavior of $\sqrt{ \sum_{i=1}^{n} \sigma_{i}^{2}(X_{-})}$, so assumption 2 is more than what is required to study the behavior of largest singular value. This result is merely a gentle introduction on how Talagrands' inequality can be used to extrapolate concentration behavior of somewhat complicated function of Gaussian ensemble, by mere knowledge of associated Lipschitz constant. In fact in a parallel work in progress, we provide a novel deviation bound for $\sigma_{1}(X_{-})$ based on spectral theorem instead of naive discretization of unite sphere.   

\section{Conclusion and Future Work}
\label{sec:conclusion}
Given $\rho \in (0,1)$ and dimension of the underlying state space $n$, we provide a uniform bound for smallest $\Gamma(n,\rho)$ such that any linear transformation $A \in \mathbb{C}^{n \times n}$ with spectral radius equal to $\rho$ satisfies $\|A^{k}\|<1$ for all $k>\Gamma(n,\rho)$. In the process we also provided first interpretable quantitative handle on the rate of decay of $\|A^{k}\|$, which was subject to various misleading speculations over last few years. Using spectral theorem for non-hermitian matrices we showed magnitude and discrepancy between algebraic and geometric multiplicity of distinct eigenvalues controls the decay rate of the norm associated with finite powers of a stable linear transformation. Moving onto regression analysis on a trajectory of dynamical system, we show how lower bounds on error require quantifying largest singular value of the data matrix: a quantitative handle is also provided by combining evolution of finite powers of linear transformation along with Talagrands inequality. We show that element-wise estimation error for OLS is an inner product and its' typical behavior can be understood via well known Littlewood-Offord problem but with i.i.d one dimensional standard normals and weights that are columns of pseudo-inverse(solution to linear constraints weighted by powers of $A$, other independent identically distributed length $N$ isotropic Gaussians, in fact can  even be dependent on Gaussian that they take an inner product with ). As a first step towards demystifying these intricate dependencies, we show how spectral theorem combined with Gaussian projection lemma allows us to spatially decouple the data matrix into lower dimensional statistically independent random dynamical systems. Various other estimates that will assist with a conclusive remark on estimation error in OLS are also collected. We would like to highlight that all the results in this paper have been derived from first principles leading to conclusive statements like: \emph{Operator norm is independent of basis choice}. Instead of naively applying results from high dimensional statistics and random matrix theory(which were developed to cater i.i.d random variable setting), we initially focused on developing non-asymptotic version of mathematical systems theory and then combined it with results from high dimensional geometry/statistics to conclude previously unknown result like: \emph{Spectral Theorem combined with Gaussian projection lemma implies that data matrix can be decomposed into low dimensional random dynamical systems which are statistically independent of each other(modulo spatially inseparable systems)}.    In a parallel line of work we have been working on spectral analysis of high dimensional data matrices, which are time realizations of LTI systems, including typical order of largest and smallest singular value of the data matrix. 
Conclusive statements on susceptibility of largest singular value(stable case) to curse of dimensionality and OLS estimation error using results on higher degree variants of Littlewood-Offord will be presented in up coming work.  
\bibliography{main}
\bibliographystyle{IEEEtran}
\newpage
\onecolumn
\section{appendix}
\label{sec:appendix}
\subsection{Lemma \ref{lm:cov-ul-bd}}

\label{subsec:pfl4}
\begin{proof}
To keep notations compact, $ \preccurlyeq  Z \preccurlyeq  $ imply lower and upper bound on the variance of some random variable $Z$.  We will first lower bound on typical size of the covariate
 Let $\lambda \in (0,1)$ implying that $[X_{-}]_{1,N}$ is normally distributed with variance:    
\begin{align}
& \nonumber \sum_{l=1}^{N} \lambda^{2(N-l)} \sum_{m=0}^{(N-l) \land (n-1)} \binom{N-l}{m}^{2} \lambda^{-2m} \geq \sum_{l=1}^{N-n} \lambda^{2(N-l)}\sum_{m=0}^{n-1} \binom{N-l}{m}^{2}\lambda^{-2m}+\sum_{l=N-n+1}^{N} \lambda^{2(N-l)} \sum_{m=0}^{N-l} \binom{N-l}{m}^{2} \\ & \nonumber \geq \sum_{l=1}^{N-n} \lambda^{2(N-l)}\sum_{m=0}^{n-1} \binom{N-l}{m}^{2}\lambda^{-2m} +4^{N} \lambda^{2N} \sum_{l=N-n+1}^{N} \frac{1}{\sqrt{16^{l} \lambda^{4l}\pi(N-l+\frac{1}{3})}} \\ & \nonumber= \underbrace{\sum_{l=1}^{N-n} \lambda^{2(N-l)}\sum_{m=0}^{n-1} \binom{N-l}{m}^{2}\lambda^{-2m}}_{C_{\lambda}(N,n)}  + \frac{4^{N} \lambda^{2N}}{\sqrt{16^{(N-n)}\lambda^{4(N-n)}}}\underbrace{\sum_{l=1}^{n} \frac{1}{\sqrt{16^{l} \lambda^{4l} \pi(n-l+\frac{1}{3})}} }_{R_{\lambda}(n)} \nonumber  = \sum_{l=1}^{N-n} \lambda^{2(N-l)}\sum_{m=0}^{n-1} \binom{N-l}{m}^{2}\lambda^{-2m} \\ &   + 4^{n} \lambda^{2n}\sum_{l=1}^{n} \frac{1}{\sqrt{16^{l} \lambda^{4l} \pi(n-l+\frac{1}{3})}} \geq \sum_{l=1}^{N-n} \lambda^{2(N-l)}\sum_{m=0}^{n-1} \binom{N-l}{m}^{2}\lambda^{-2m}  + \frac{4^{n}}{\lambda^2} \lambda^{2n}\sum_{l=1}^{n} \frac{1}{\sqrt{16^{l}  \pi(n-l+\frac{1}{3})}}  
\end{align}
Therefore, for every $ \lambda \in (\frac{1}{2},1)  $ typical behavior(standard deviation) of $[X_{-}]_{1,N} \geq \Omega_{\lambda}(e^{n})+C_{\lambda}(N,n)$, where $C_{\lambda}(N,n)$ can be lower bounded as:
\begin{align}
& \nonumber C_{\lambda}(N,n)= \lambda^{2N}\sum_{l=1}^{N-n} \sum_{m=0}^{n-1} \binom{N-l}{m}^{2}\lambda^{-2(l+m)}= \lambda^{2N}\sum_{l=1}^{N-n} \sum_{m=0}^{n-1} \frac{\big(N-l\big)^{2m} \prod_{p=1}^{m-1}\big(1-\frac{p}{N-l}\big)^{2}}{\lambda^{2(l+m)}m!m!} \\ & \nonumber = \lambda^{2N}  \sum_{m=0}^{n-1} \frac{1}{\lambda^{2m}m!m!}\Bigg(\frac{\big(N-1\big)^{2m} \prod_{p=1}^{m-1}\big(1-\frac{p}{N-1}\big)^{2}}{\lambda^{2(1)}} + \ldots  +\frac{\big(N-n+1\big)^{2m} \prod_{p=1}^{m-1}\big(1-\frac{p}{N-n+1}\big)^{2}}{\lambda^{2(N-n+1)}} \Bigg)  \\ & \label{eq:bplbd} \geq \lambda^{2N}  \sum_{m=0}^{n-1} \frac{1}{\lambda^{2m}m!m!} \sum_{l=1}^{N-n} \frac{\big(N-l\big)^{2m}}{\lambda^{2l}} \exp{\bigg(-\frac{m(m-1)}{N-l}\bigg)}\exp{\Bigg(-O\bigg(2\frac{m^{3}}{(N-l)^{2}}\bigg)\Bigg)}. 
\end{align}
Upper bounding typical size of the covariate: $[X_{-}]_{1,N}$ is normally distributed with variance:    
\begin{align}
& \nonumber \sum_{l=1}^{N} \lambda^{2(N-l)} \sum_{m=0}^{(N-l) \land (n-1)} \binom{N-l}{m}^{2} \lambda^{-2m}=\sum_{l=1}^{N-n} \lambda^{2(N-l)}\sum_{m=0}^{n-1} \binom{N-l}{m}^{2}\lambda^{-2m}+\sum_{l=N-n+1}^{N} \lambda^{2(N-l)} \sum_{m=0}^{N-l} \binom{N-l}{m}^{2} \lambda^{-2m} \\ & \nonumber \leq \lambda^{2(N-n+1)} \sum_{l=1}^{N-n} \frac{1}{\lambda^{2l}}\sum_{m=0}^{n-1} \frac{\big(N-l\big)^{2m} \prod_{p=1}^{m-1}\big(1-\frac{p}{N-l}\big)^{2}}{m!m!} + \frac{4^{N}}{\lambda^{2}}\sum_{l=N-n+1}^{N}  \frac{1}{\sqrt{16^{l}\pi(N-l+\frac{1}{4})}} \\ & \label{eq:bpdubd}\leq  \lambda^{2(N-n+1)} \sum_{m=0}^{n-1} \frac{1}{m!m!} \sum_{l=1}^{N-n} \frac{(N-l)^{2m}}{\lambda^{2l}} \exp{\bigg(-\frac{m(m-1)}{N-l}\bigg)} + \frac{4^{n}}{\lambda^{2}}\sum_{l=1}^{n}  \frac{1}{\sqrt{16^{l}\pi(n-l+\frac{1}{4})}}.
\end{align}
Therefore for all $N \geq n$ and $\lambda \in (\frac{1}{2},1)$
\begin{equation}
  C_{\lambda}(N,n)+\Omega(4^{n} \lambda^{2n})(4^{n} \lambda^{2n}) \leq  [X_{-}]_{1,N} \leq C_{\lambda}(N,n)+O_{\lambda}(4^{n}),
\end{equation}
combined with analysis related to birthday paradox problem we have:
\begin{align}
    & \nonumber \lambda^{2N}  \sum_{m=0}^{n-1} \lambda^{-2m}  \sum_{l=1}^{N-n} \frac{1}{m!m!} \frac{\big(N-l\big)^{2m}}{\lambda^{2l}} \exp{\bigg(-\frac{m(m-1)}{N-l}\bigg)}\exp{\Bigg(-O\bigg(2\frac{m^{3}}{(N-l)^{2}}\bigg)\Bigg)}+O(4^{n} \lambda^{2n}) \\ & \nonumber \preccurlyeq  [X_{-}]_{1,N}  \preccurlyeq \lambda^{2(N-n+1)} \sum_{m=0}^{n-1}  \sum_{l=1}^{N-n} \underbrace{\frac{1}{m!m!} \frac{(N-l)^{2m}}{\lambda^{2l}} \exp{\bigg(-\frac{m(m-1)}{N-l}\bigg)}}_{\zeta_{m,l}(N,\lambda)} +O_{\lambda}(4^{n})
\end{align}
Therefore variance of first row is:
\begin{align}
& \nonumber \sum_{i>n}^{N}[X_{-}]_{1,i} \preccurlyeq \sum_{i>n}^{N}  \sum_{m=0}^{n-1} \lambda^{2(i-n+1)} \sum_{l=1}^{i-n} \zeta_{m,l}(i,\lambda)+(N-n+1)O_{\lambda} (4^{n}) \\ & \nonumber \sum_{i>n}^{N}[X_{-}]_{1,i} \succcurlyeq  \sum_{i>n}^{N} \sum_{m=0}^{n-1} \lambda^{2(i-m)} \sum_{l=1}^{i-n} \zeta_{m,l}(i,\lambda) \exp{\Bigg(-O\bigg(2\frac{m^{3}}{(i-l)^{2}}\bigg)\Bigg)}+(N-n+1)\Omega(4^{n} \lambda^{2n})    
\end{align}

\end{proof}
\subsection{Theorem \ref{thm:invcovneg2ndmoment} }
\label{subsec:pfinvcovneg2ndmoment}
\begin{proof}
Let $v_{j}:=\big(X_{-}X_{-}^{*}\big)^{-1}e_{j}$, in columns format we have $X_{-}^{*}=[y_1,\ldots,y_{n}]$, $X_{-}^{*} e_{j}=y_{j}$, $n_{j}=$span$[y_{i}]_{i \neq j}$.
\begin{align}
 & \label{eq:deltajk} \bigg\langle X_{-}^{*} \big(X_{-}X_{-}^{*}\big)^{-1}e_{j},X_{-}^{*}e_{k}\bigg \rangle =\delta_{j}(k)
\end{align}
Notice that inner product of $X_{-}^{*} \big(X_{-}X_{-}^{*}\big)^{-1}e_{j}$ with $X_{-}^{*}e_{k}=:y_{k}$ for $k \neq j$ is zero; so for each $j \in [n]$ , $X_{-}^{*} \big(X_{-}X_{-}^{*}\big)^{-1}e_{j}$ is orthogonal to $n_{j}$. Since,
\begin{equation}
    X_{-}^{*}v_{j}=v_{j,1}y_1+\ldots+v_{j,n}y_{n},
\end{equation}
1) and 2) readily follows. As   $X_{-}^{*} \big(X_{-}X_{-}^{*}\big)^{-1}e_{j}$ is orthogonal to $n_{j}$, according to  Corollary \ref{cor:vimpescpnegmomstuck} 
\begin{equation}
    \big\langle X_{-}^{*} \big(X_{-}X_{-}^{*}\big)^{-1}e_{j},X_{-}^{*}e_{j}\big\rangle = \big\langle X_{-}^{*} 
    \big(X_{-}X_{-}^{*}\big)^{-1}e_{j},P_{n_j^{\perp}}(X_{-}^{*}e_{j})\big\rangle
\end{equation}
Furthermore, using \eqref{eq:deltajk} and properties of orthogonal projection:
\begin{align}
     \nonumber 1 &=\big\langle X_{-}^{*} \big(X_{-}X_{-}^{*}\big)^{-1}e_{j},X_{-}^{*}e_{j}\big\rangle = \big\langle X_{-}^{*} 
    \big(X_{-}X_{-}^{*}\big)^{-1}e_{j},P_{n_j^{\perp}}(X_{-}^{*}e_{j})\big\rangle  \\ &= \nonumber \big\langle P_{n_j^{\perp}}\big[X_{-}^{*} \big(X_{-}X_{-}^{*}\big)^{-1}e_{j}\big],X_{-}^{*}e_{j}\big\rangle= \sum_{i\neq j}^{n} v_{j,i}\big\langle \underbrace{P_{n_j^{\perp}}y_{i}}_{=0}, y_{j}\big \rangle \\ & \hspace{15pt} +\nonumber v_{j,j}\langle P_{n_j^{\perp}} y_{j} , y_{j}\rangle  =v_{j,j}\langle P_{n_j^{\perp}} y_{j}, P_{n_j^{\perp}} y_{j}\rangle =v_{j,j} \|P_{n_j^{\perp}} y_{j}\|^{2} \\ & \nonumber =v_{j,j}d^{2}(y_{j},n_{j}). \hspace{10pt} \textit{Therefore,} \hspace{5pt} v_{j,j}=\frac{1}{d^{2}(y_{j},n_{j})}
\end{align}
\end{proof}

\end{document}